\documentclass[a4paper]{amsart}
\usepackage{a4wide,amsfonts,amssymb,amsmath,enumerate,amscd,color,tikz,tikz-3dplot,tikz-cd,todonotes,ulem}
\usepackage[english]{babel}
\usepackage{hyperref}
\allowdisplaybreaks
\usepackage{comment}
\newtheorem{lem}{Lemma}[section]
\newtheorem{defi}[lem]{Definition}
\newtheorem{theo}[lem]{Theorem}
\newtheorem{cor}[lem]{Corollary}
\newtheorem{rem}[lem]{Remark}
\newtheorem{prop}[lem]{Proposition}

\newcommand{\Hmmc}[2]{\leavevmode{\marginpar{\tiny%
\color{#1}$\hbox to 0mm{\hspace*{-0.5mm}$\leftarrow$\hss}%
\vcenter{\vrule depth 0.1mm height 0.1mm width \the\marginparwidth}%
\hbox to
0mm{\hss$\rightarrow$\hspace*{-0.5mm}}$\\\relax\raggedright #2}}}
\newcommand{\Hmm}[1]{\Hmmc{blue}{#1}}
\newcommand{\Hmmo}[1]{\Hmmc{red}{#1}}

\def\cred{\color{red}}

\def\cblack{\color{black}}

\def\ol{\overline}

\def\incl{\hookrightarrow}
\def\reals{\mathbb{R}}
\def\nat{\mathbb{N}}

\def\ga{\Gamma}
\def\gap{\ga_{\Phi}}
\def\gat{\ga_{\!t}}
\def\gan{\ga_{\!n}}
\def\gatp{\ga_{\!t,\Phi}}

\def\gatpchi{\ga_{\!t,\Phi_{\chi}}}

\def\ganp{\ga_{\!n,\Phi}}

\def\ganpchi{\ga_{\!n,\Phi_{\chi}}}

\def\om{\Omega}
\def\omp{\om_{\Phi}}
\def\omb{\ol{\om}}
\def\ombp{\omb_{\Phi}}
\def\n{\mathrm{n}}

\def\mcX{\mathcal{X}}
\def\sfL{\mathsf{L}}
\def\sfH{\mathsf{H}}
\def\sfC{\mathsf{C}}

\def\sfR{\mathsf{R}}
\def\sfD{\mathsf{D}}

\renewcommand{\L}{\sfL}
\renewcommand{\H}{\sfH}

\newcommand{\C}{\sfC}
\newcommand{\R}{\sfR}
\newcommand{\D}{\sfD}

\newcommand{\V}{\mathsf{V}}
\newcommand{\Pc}{\Phi_\chi}

\newcommand{\eps}{\varepsilon}

\DeclareMathOperator{\A}{A}
\DeclareMathOperator{\cA}{\mathcal{A}}
\DeclareMathOperator{\cL}{\mathcal{L}}

\DeclareMathOperator{\p}{\partial}
\DeclareMathOperator{\id}{id}
\DeclareMathOperator{\adj}{adj}
\DeclareMathOperator{\sym}{sym}

\DeclareMathOperator{\tr}{tr}
\DeclareMathOperator{\TAc}{\mathcal{T}_{\chi}}
\DeclareMathOperator{\Tc}{T_{\chi}}
\DeclareMathOperator{\TAtc}{\mathcal{T}_{\bar\chi}}
\DeclareMathOperator{\Ttc}{T_{\bar\chi}}
\DeclareMathOperator{\Jc}{\mathcal{I}_{\chi}}
\DeclareMathOperator{\Jtc}{\mathcal{I}_{\bar\chi}}
\DeclareMathOperator{\symtr}{symtr}

\DeclareMathOperator{\na}{\nabla}

\DeclareMathOperator{\rot}{rot}

\DeclareMathOperator{\divergence}{div}
\def\div{\divergence}

\newcommand{\wt}[1]{\widetilde{#1}}
\newcommand{\wh}[1]{\widehat{#1}}
\newcommand{\norm}[1]{|#1|}

\newcommand{\scp}[2]{\langle#1,#2\rangle}
\newcommand{\bscp}[2]{\big\langle#1,#2\big\rangle}
\newcommand{\Bscp}[2]{\Big\langle#1,#2\Big\rangle}

\title[Shape Derivatives of the  Eigenvalues of the de Rham Complex: Part II]
{Shape Derivatives\\ of the  Eigenvalues of the de Rham Complex\\  
for Lipschitz deformations and variable coefficients:\\
Part II}
\author{Pier Domenico Lamberti}
\author{Dirk Pauly}
\author{Michele Zaccaron}
\address{Dipartimento di Tecnica e Gestione dei Sistemi Industriali, University of Padova, Italy}
\email[Pier Domenico Lamberti]{lamberti@math.unipd.it}
\address{Institut f\"ur Analysis, Technische Universit\"at Dresden, Germany}
\email[Dirk Pauly]{dirk.pauly@tu-dresden.de}
\address{Aix Marseille Univ , CNRS, Centrale Marseille, Institut Fresnel, Marseille, France}
\email[Michele Zaccaron]{michele.zaccaron@univ-amu.fr}
\date{\today}
\thanks{{\it Corresponding Author}: Michele Zaccaron}

\begin{document}


\begin{abstract}

In this second part of our series of papers, we develop an abstract
framework suitable for de Rham complexes that depend on a parameter
belonging to an arbitrary Banach space. Our primary focus is on spectral
perturbation problems and the differentiability of eigenvalues with
respect to perturbations of the involved parameters.

As a byproduct, we provide a proof of the celebrated Hellmann-Feynman
theorem for both simple and multiple eigenvalues of suitable families of
self-adjoint operators in Hilbert spaces, even when these operators
depend on possibly infinite-dimensional parameters.

We then apply this abstract machinery to the de Rham complex in three
dimensions, considering mixed boundary conditions and non-constant
coefficients. In particular, we derive Hadamard-type formulas for
Maxwell and Helmholtz eigenvalues. First, we compute the derivatives
under minimal regularity assumptions—specifically, Lipschitz
regularity—on both the domain and the perturbation, expressing the
results in terms of volume integrals. Second, under more regularity
assumptions on the domains, we reformulate these formulas in terms of
surface integrals.

\end{abstract}


\maketitle
\setcounter{tocdepth}{3}
{\small
\tableofcontents}


\section{Introduction}

We continue our analysis from the first part
where we have discussed the eigenvalue problem for the de Rham complex 
on variable (transplanted) domains in a general setting.

We use our notations from the first part.
Throughout this paper, let $(\om,\gat)$ be a bounded weak Lipschitz pair, where $\Omega$ is a bounded open set in $\mathbb{R}^3$ with 
boundary $\Gamma$. Moreover, let  $\gan :=\Gamma \setminus \overline\gat$ and  $\nu$, $\eps$, $\mu$, be admissible coefficients.
These assumptions will be sharpened later. 

In our first part we have shown that the de Rahm complex has countably many eigenvalues 
\begin{align}
\label{eq:intro:eigenval1}
0<\lambda_{\ell,\Phi,1}\leq\lambda_{\ell,\Phi,2}\leq\dots
\leq\lambda_{\ell,\Phi,k-1}\leq\lambda_{\ell,\Phi,k}\leq\dots\to\infty,\qquad
\ell\in\{0,1\}.
\end{align}
Here $\Phi:\om\to\Phi(\om)=\omp$ is a bi-Lipschitz transformation that 
models the variations of the domain $\om$,
and the indices $0$ and $1$ indicate Helmholtz and Maxwell eigenvalues, respectively. More precisely,  $\lambda_{0,\Phi,k}$ are the eigenvalues of the boundary value problem
\begin{equation}\label{helmprobl}
\begin{aligned}
-\nu^{-1}\div\eps\na u&=\lambda_{0}u
&\text{in }&\omp,\\
u&=0
&\text{on }&\gatp,\\
\n\cdot\eps\na u&=0
&\text{on }&\ganp ,
\end{aligned}
\end{equation}
and   $\lambda_{1,\Phi,k}$ are the eigenvalues of 
\begin{align}\label{maxprobl}
\begin{aligned}
\eps^{-1}\rot\mu^{-1}\rot E&=\lambda_{1}E
&\text{in }&\omp,\\
\n\times E&=0
&\text{on }&\gatp,\\
\n\times\mu^{-1}\rot E&=0
&\text{on }&\ganp .
\end{aligned}
\end{align}

Here 
$$\gap:=\Phi(\ga),\qquad
\gatp:=\Phi(\gat),\qquad
\ganp:=\Phi(\gan)$$
and $(\omp, \gatp)$ is also a weak Lipschitz pair.

Concerning the transformation $\Phi$, we recall that 
$\Phi\in\C^{0,1}(\omb,\ombp)$ and $\Phi^{-1}\in\C^{0,1}(\ombp,\omb)$ with\footnote{Note that the Jacobian determinant of a bi-Lipschitz diffeomorphism has a constant sign on the connected components of the domain, see \cite[Lemma~6.7]{res}, hence it is not restrictive to assume that it is positive almost everywhere.} 
$$J_{\Phi}=\Phi'=(\na\Phi)^{\top},\qquad
{\rm ess}\inf\det J_{\Phi}>0.$$ 
Such  bi-Lipschitz transformations will be called admissible and we write
$$\Phi\in\cL(\om).$$
Note that $\cL(\om)$ is understood as a subset of  $\C^{0,1}(\omb,\ombp)$ and  $\C^{0,1}(\omb,\ombp)$
is endowed with its standard Lipschitz norm (the sum of the sup norm of $|\Phi|$ and the Lipschitz constant of $|\Phi|$). 
For $\Phi\in\cL(\om)$ the inverse and adjunct matrix of $J_{\Phi}$ shall be denoted by
$$J_{\Phi}^{-1},\qquad
\adj J_{\Phi}:=(\det J_{\Phi})J_{\Phi}^{-1},$$
respectively.
We denote the composition with $\Phi$
by tilde, i.e., for any tensor field $\psi$ we define 
$$\wt{\psi}:=\psi\circ\Phi.$$

Let us fix a certain index $k$ of the sequence of eigenvalues and consider 
$$\lambda_{0,\Phi}:=\lambda_{0,\Phi,k},\qquad
\lambda_{1,\Phi}:=\lambda_{1,\Phi,k}.$$
Then by the conclusion of the Part I of this series of papers
the eigenvalues are given by the Rayleigh quotients of the respective eigenfields
$u$, $E$, $H$, or  
\begin{align*}
\tau^{0}_{\Phi}u=u\circ\Phi,\qquad
\tau^{1}_{\Phi}E=J_{\Phi}^{\top}(E\circ\Phi),\qquad
\tau^{2}_{\Phi}H=(\det J_{\Phi})J_{\Phi}^{-1}(H\circ\Phi),
\end{align*}
this is:
\begin{align}
\begin{aligned}
\label{eq:intro:eigenval2}
\lambda_{0,\Phi}
&=\frac{\scp{\eps\na_{\gat,\Phi}u}{\na_{\gat,\Phi}u}_{\L^{2}(\omp)}}
{\scp{\nu u}{u}_{\L^{2}(\omp)}}
&
\text{and}\quad
\lambda_{1,\Phi}
&=\frac{\scp{\mu^{-1}\rot_{\gat,\Phi}E}{\rot_{\gat,\Phi}E}_{\L^{2}(\omp)}}
{\scp{\eps E}{E}_{\L^{2}(\omp)}}\\
&=\frac{\scp{\eps_{\Phi}\na_{\gat}\tau^{0}_{\Phi}u}{\na_{\gat}\tau^{0}_{\Phi}u}_{\L^{2}(\om)}}
{\scp{\nu_{\Phi}\tau^{0}_{\Phi}u}{\tau^{0}_{\Phi}u}_{\L^{2}(\om)}}
&
&=\frac{\scp{\mu_{\Phi}^{-1}\rot_{\gat}\tau^{1}_{\Phi}E}{\rot_{\gat}\tau^{1}_{\Phi}E}_{\L^{2}(\om)}}
{\scp{\eps_{\Phi}\tau^{1}_{\Phi}E}{\tau^{1}_{\Phi}E}_{\L^{2}(\om)}}\\
&=\frac{\scp{\nu^{-1}\div_{\gan,\Phi}\eps H}{\div_{\gan,\Phi}\eps H}_{\L^{2}(\omp)}}
{\scp{\eps H}{H}_{\L^{2}(\omp)}}\\
&=\frac{\scp{\nu_{\Phi}^{-1}\div_{\gan}\eps_{\Phi}\tau^{1}_{\Phi}H}{\div_{\gan}\eps_{\Phi}\tau^{1}_{\Phi}H}_{\L^{2}(\om)}}
{\scp{\eps_{\Phi}\tau^{1}_{\Phi}H}{\tau^{1}_{\Phi}H}_{\L^{2}(\om)}}
\end{aligned}
\end{align}
Here
\begin{align*}
\eps_{\Phi}&=\tau^{2}_{\Phi}\eps\tau^{1}_{\Phi^{-1}}
=(\det J_{\Phi})J_{\Phi}^{-1}\wt{\eps}J_{\Phi}^{-\top},
&
\nu_{\Phi}&=\tau^{3}_{\Phi}\nu\tau^{0}_{\Phi^{-1}}
=(\det J_{\Phi})\wt{\nu},\\
\mu_{\Phi}&=\tau^{2}_{\Phi}\mu\tau^{1}_{\Phi^{-1}}
=(\det J_{\Phi})J_{\Phi}^{-1}\wt{\mu}J_{\Phi}^{-\top}
\end{align*}
and
$$\na\tau^{0}_{\Phi}=\tau^{1}_{\Phi}\na,\qquad
\rot\tau^{1}_{\Phi}=\tau^{2}_{\Phi}\rot,\qquad
\div\tau^{2}_{\Phi}=\tau^{3}_{\Phi}\div.$$
Clearly, we have in $\L^{2}_{\nu}(\omp)$ resp. $\L^{2}_{\eps}(\omp)$
\begin{align*}
-\nu^{-1}\div_{\ganp}\eps\na_{\gatp}u&=\lambda_{0,\Phi}u,
&
\eps^{-1}\rot_{\ganp}\mu^{-1}\rot_{\gatp}E&=\lambda_{1,\Phi}E,\\
-\na_{\gatp}\nu^{-1}\div_{\ganp}\eps H&=\lambda_{0,\Phi}H
\intertext{and in $\L^{2}_{\nu_{\Phi}}(\om)$ resp. $\L^{2}_{\eps_{\Phi}}(\om)$}
-\nu_{\Phi}^{-1}\div_{\gan}\eps_{\Phi}\na_{\gat}\tau^{0}_{\Phi}u
&=\lambda_{0,\Phi}\tau^{0}_{\Phi}u,
&
\eps_{\Phi}^{-1}\rot_{\gan}\mu_{\Phi}^{-1}\rot_{\gat}\tau^{1}_{\Phi}E
&=\lambda_{1,\Phi}\tau^{1}_{\Phi}E,\\
-\na_{\gat}\nu_{\Phi}^{-1}\div_{\gan}\eps_{\Phi}\tau^{1}_{\Phi}H
&=\lambda_{0,\Phi}\tau^{1}_{\Phi}H.
\end{align*}
Note that the eigenvalues $\lambda_{\ell,\Phi,k}$ are depending
not only on $\Phi$ (shape of the domain) but also on the mixed boundary conditions imposed on $\gat$ and $\gan$
and on the coefficients $\eps$, $\mu$, and $\nu$, which we do not indicate explicitly in our notations, i.e.,
$$\lambda_{0,\Phi,k}
=\lambda_{0,\Phi,k}(\om,\gat,\eps,\nu),\qquad
\lambda_{1,\Phi,k}
=\lambda_{1,\Phi,k}(\om,\gat,\eps,\mu).$$


In the first part of this series of papers, we have found Hadamard-type formulas for the derivatives of $\lambda_{0,\Phi}$
 and $\lambda_{1,\Phi}$ with respect to $\Phi$ under the assumption that the corresponding (pull-backs of the)  eigenfunctions are differentiable with respect to $\Phi$. The formulas were obtained by differentiating  the Rayleigh quotients \eqref{eq:intro:eigenval2} with respect to $\Phi$. In particular, as for the classical Hellmann-Feynman Theorem, one realizes that the derivatives  of the eigenvectors with respect to $\Phi$ cancel out and do not appear in the formulas. 
 We note that the differentiability  assumption  for the eigenfunctions typically requires that the eigenvalue under consideration is simple, since multiple eigenvalues may bifurcate   into several branches, thus preventing their differentiability with respect to $\Phi$.   In the case of a family of transformations $\Phi_{\chi}$ depending smoothly on a real parameter $\chi$, one may relabel the eigenvalues and repeat the same formal argument for a branch of eigenvalues and eigenfunctions that are assumed to be differentiable with respect to $\chi$. In  general, this relabelling  is not possible for parameters $\chi$ belonging to a multi-dimensional space. 
 
In this second part of this series of papers  we consider transformations $\Phi_{\chi}$ depending on a parameter $\chi$ belonging to an arbitrary Banach space, possibly infinite dimensional, and  analyze this problem in detail. More precisely, we prove a differentiability result for the elementary symmetric functions of the eigenvalues, we provide Hadamard-type formulas for the corresponding directional derivatives and  discuss the bifurcation of multiple eigenvalues by proving a Rellich-type theorem with corresponding formulas for the derivatives.  First, the derivatives are computed under minimal, that is Lipschitz, regularity assumptions on $\Omega$ and $\Phi$ and are formulated in terms of volume integrals. Second, under more general regularity assumptions on $\Omega$, we express those formulas in terms of surface integrals. In particular, this is done by assuming the validity of the Gaffney inequality.  We call the formulas for the Maxwell eigenvalues `Hirakawa's formulas' since, as far as we know, they first appeared in \cite{hira}. 
These results are contained in Section \ref{applications} and are obtained as an application of a number of  theoretical results proved in Section \ref{sec:shape2sound}  
where we provide an abstract framework suitable for de Rham complexes depending on a possibly infinite-dimensional parameter $\chi$. We note that, as a bypass product of our analysis, we provide a rigorous proof of the Hellmann-Feynman Theorem for these eigenvalue problems, see Remark~\ref{rulethumb}.

\section{Perturbation Theory in an abstract framework}
\label{sec:shape2sound}


\label{perturbtheory}

Let $\H_0$ and $\H_1$ be fixed Hilbert spaces. We assume that, besides the scalar products $\scp{\cdot }{\cdot}_{\H_{0}}$ and  $\scp{\cdot }{\cdot}_{\H_{1}}$, we are also given two families of scalar products $\scp{\cdot }{\cdot}_{0,\chi } $ and $\scp{\cdot }{\cdot}_{1,\chi } $ defined on the vector spaces $\H_0$ and $\H_1$ respectively, that depend on a parameter $\chi$ belonging to some open subset $\mcX$ of a fixed real Banach space $X$.  
We assume that the norms associated with $\scp{\cdot }{\cdot}_{0,\chi } $ and $\scp{\cdot }{\cdot}_{1,\chi } $ are equivalent to those associated with  $\scp{\cdot }{\cdot}_{\H_{0}}$ and  $\scp{\cdot }{\cdot}_{\H_{1}}$ respectively. The resulting Hilbert spaces will be denoted by 
$\H_{0,\chi}$ and $\H_{1,\chi}$ and the scalar products $\scp{\cdot }{\cdot}_{0,\chi } $ and $\scp{\cdot }{\cdot}_{1,\chi } $ will also be denoted by 
 $\scp{\cdot }{\cdot}_{\H_{0,\chi}}$ and  $\scp{\cdot }{\cdot}_{\H_{1,\chi}}$.

We assume that for any $\chi\in\mcX$ we are given a densely defined and closed linear operator 
 $\A_{\chi}:D(\A_{\chi})\subset\H_{0,\chi}\to\H_{1,\chi}$ 
with domain of definition $D(\A_{\chi})$. 

Our main assumption is that $D(\A_{\chi})$ does not depend on $\chi$. To emphasise it, we denote by $\V$ the space $D(\A_{\chi})$ by setting
$\V=D(\A_{\chi})$. Since  $D(\A_{\chi})$ is naturally endowed with the graph norm of $\A_{\chi}$ which depends on $\chi$, we assume that those graph norms are equivalent and, unless otherwise indicated, we understand that $\V$ is endowed with one fixed graph norm chosen among them (for many  of our arguments, the specific choice is not relevant).    In particular, the topological dual $\V'$ is well-defined, as well as the analogous space $\V'_{\sharp}$ of
antilinear continuous maps from $\V$ to $\mathbb{C}$ (recall that $\V'$ and $\V'_{\sharp}$ are homeomorphic by conjugation). Note that since $\A_{\chi}$ is closed, $\V$  turns out to be complete, hence it is a Hilbert space continuously embedded into $\H_0$.

Recall that the Hilbert space adjoint $\A^{*}_{\chi}:D(\A^{*}_{\chi})\subset\H_{1,\chi}\to\H_{0,\chi}$
is well defined and characterised by
$$\forall\,x\in \V\quad
\forall\,y\in D(\A^{*}_{\chi})\qquad
\scp{\A_{\chi} x}{y}_{\H_{1,\chi}}=\scp{x}{\A^{*} y}_{\H_{0,\chi}}.$$
The operator  $\A^{*}_{\chi}$ is also  densely defined and  closed\footnote{Note that in principle $  D(\A^{*}_{\chi})$ may depend on $\chi$ but this does not affect  our analysis}.  Recall that  by $\cA_{\chi}$ we denote the reduced operator associated with $\A_{\chi}$. 

It is convenient to set $\Tc=\A^*_{\chi}\A_{\chi}$ and to recall from the first part that $\Tc$ is a non-negative self-adjoint operator densely defined 
in  $\H_{0,\chi}$.
 We shall always  assume that the embedding $D(\cA_{\chi})  \incl\H_{0} $ is compact.   This implies that 
 $$\sigma(\Tc   )\setminus\{0\}
=\{\lambda_{\chi, k}\}_{k\in\nat}
\subset(0,\infty)$$
with eigenvalues 
$0<\lambda_{\chi, 1}\le \lambda_{\chi , 2}\le \dots\le\lambda_{\chi, k-1}\le \lambda_{\chi, k}\le\dots\to\infty$
of finite multiplicity. Here we agree to repeat each eigenvalue in the sequence as many times as its multiplicity.

In order  to study the dependence of the resolvent of $\Tc$ upon variation of  $\chi$, we need to represent it  in a suitable way.  To do so, we  consider  the operator  $\TAc:\V\to\V'_{\sharp}$ which takes any $u\in \V$ to the operator $\TAc u$ defined by 
$$
\scp{\TAc u }{v}= \scp{\A_{\chi} u}{\A_{\chi} v}_{\H_{1,\chi}}
$$
for all $v\in V$.  In a similar way, we consider the operator $\Jc:\H_0\to\V'_{\sharp}$ which takes any $u\in \H_0$ to the operator $\Jc u$ defined by 
$$
\scp{\Jc  u }{v}= \scp{ u}{ v}_{\H_{0,\chi}}
$$
for all $v\in V$.    Then we can prove the following

\begin{lem}\label{res0}
If  $\zeta \in \mathbb{C}\setminus \sigma (\Tc)$  then the  operator $ \TAc-\zeta \Jc  $ is invertible and 
\begin{equation}\label{res1}
\left(\TAc-\zeta \Jc   \right)^{-1}\circ \Jc=(\Tc-\zeta I)^{-1}.
\end{equation}
\end{lem}

\begin{proof} We begin by observing that given $F\in \V_{\sharp}'$ the equation 
$$
\TAc u-\zeta \Jc  u= F
$$
in the unknown $u\in \V$ can be formulated in the equivalent form 
$$
\scp{\A_{\chi} u}{\A_{\chi} v}_{\H_{1,\chi}}-\zeta \scp{u}{v}_{\H_{0,\chi}}=Fv
$$
for all $v\in \V$. It follows that if $\TAc u-\zeta \Jc  u= 0$ for some $u\in \V$ then $u=0$ otherwise $\zeta $ would be an eigenvalue of $\Tc$. Thus $\TAc -\zeta \Jc $ is injective. To prove that it is surjective, we first note that by the Riesz Theorem any functional $F\in  \V_{\sharp}'$ can be represented in the form 
$ F=\TAc w + \Jc  w $ for some $w\in \V$. Thus, we have to prove that given $w\in \V$ there exists $u\in \V$ such that 
\begin{equation}\label{res2}
\TAc u-\zeta \Jc  u= \TAc w+ \Jc  w
\end{equation}
Since $\zeta$ is in the resolvent set of $\Tc$ there exists $u_{w} \in V$ such that 
$$
\TAc u_{w}-\zeta \Jc  u_{w}= \Jc w.
$$
It follows that the solution $u$ to equation \eqref{res2} is given by $u=w+(\zeta +1)u_{w}$. 

The equality in \eqref{res1} is straightforward. 
\end{proof}

\begin{rem} Although the domain $\H_0$ of $(\Tc-\zeta I)^{-1}$ does not depend on $\chi$, the domain of  $\Tc-\zeta I$ may change when $\chi$ is perturbed.
The advantage of the representation formula \eqref{res1}  consists in the fact that the domains and codomains of the operators 
$\TAc-\zeta \Jc$ and $\Jc$ do not depend on $\chi$. This allows to 
easily prove regularity results for the dependence of  $(\Tc-\zeta I)^{-1}$ on $\chi$.  
\end{rem}

In order to study the regularity properties of the eigenvalues $\lambda_{\chi, k}$ we need appropriate regularity assumptions on the families of operators $\A_{\chi}$ and scalar products $\scp{\cdot}{\cdot}_{0,\chi}$, $\scp{\cdot}{\cdot}_{1,\chi}$. \\

The regularity assumption in the following definition is given with respect to the real structure of the Banach spaces involved, in which case differentials are required to be $\reals$-linear functionals (not necessarily $\mathbb{C}$-linear).  In particular, by a function of class $C^{\omega}$ we mean a  real-analytic function, i.e.,  the function is $C^{\infty}$  and locally representable by Taylor's series.

\begin{defi}\label{reg}
{\bf (Regularity assumption)}  We say that the family of operators and scalar products under consideration are of class $\C^r$ with $r\in \nat_0\cup\{\infty\}\cup\{\omega\}$ 
 if 
both the map from $\mathcal{X}$ to ${\mathcal{L}}(V, V'_{\sharp})$ which takes any $\chi\in \mathcal{X}$ to $\TAc$, and the map from $\mathcal{X}$ to ${\mathcal{L}}(\H_0, V'_{\sharp})$ which takes any $\chi\in \mathcal{X}$ to $\Jc$ are of class $\C^r$. 
\end{defi}

We also  recall the definition of gap between operators from \cite[Ch. IV]{K1976}

\begin{defi}\label{gap} Given a Banach space $Z$ and two closed subspaces $M,N$ of $Z$ the gap between $M,N$ is defined by 
$$
\hat\delta (M,N)= \max \{ \delta (M, N), \delta (N,M) \}
$$
where 
$$
\delta (M,N)= \sup_{u\in S_M }{\rm dist} (u, N),\ \ \delta (N,M)= \sup_{u\in S_N }{\rm dist} (u, M)
$$
and $S_M,S_N$ denote sets of unit vectors in $M$ and $N$ respectively.  

Given two closed operators $T$, $S$ acting between two fixed  Banach spaces $E,F$, the gap between $T$ and $S$ is defined by 
$$
\hat\delta(T,S)=\hat\delta (G(T),G(S))
$$
where $G(T)$ and $G(S)$ denote the graphs $T$ and $S$ respectively (considered as closed subspaces of the product space $Z=E\times F$).
\end{defi}

\begin{lem}\label{gapo1}  Assume that the family of operators and scalar products are of class  $\C^0$. Let $\bar\chi\in \mathcal{X}$ be fixed. Then 
$\hat \delta (\Tc,\Ttc) =o(1)$ as $\chi\to \bar \chi$.
\end{lem}

\begin{proof} By \cite[Thm.~2.17, p.204]{K1976} it follows that $\hat \delta (\Tc,\Ttc) \le 4\hat \delta (\Tc+I,\Ttc +I)$ and by \cite[Thm.~2.20, p.205]{K1976}
$\delta (\Tc+I,\Ttc +I)=\delta ((\Tc+I)^{-1},(\Ttc +I)^{-1})$. Moreover, by \cite[Thm.~2.14, p.203]{K1976} we have that 
$\delta ((\Tc+I)^{-1},(\Ttc +I)^{-1})\le \|  (\Tc+I)^{-1} - (\Ttc +I)^{-1} \|$. Thus, by formula \eqref{res1}  we have 
$\hat \delta (\Tc,\Ttc)  \le 4
\|  \left(\TAc+ \Jc   \right)^{-1}\circ \Jc   -      \left(\TAtc + \Jtc   \right)^{-1}\circ \Jtc   \| $. By our regularity assumption and by observing that the inversion (for functions between fixed spaces) is an analytic operator, we conclude that $
\|  \left(\TAc+ \Jc   \right)^{-1}\circ \Jc   -      \left(\TAtc + \Jtc   \right)^{-1}\circ \Jtc   \| =o(1)$ as $\chi\to \bar \chi$ and the proof is complete. 
\end{proof}

Let $\chi\in \mathcal X$ and let   $\sigma (\Tc)=\Sigma'_{\chi}\cup \Sigma_{\chi} ''$ where $\Sigma' _{\chi}\cap \Sigma_{\chi} '' =\emptyset$. Assume that  $\Sigma'_{\chi}$ and $\Sigma_{\chi} ''$ are (closed) sets  separated by a simple, closed, rectifiable curve $\Gamma $ in the sense  that $\Sigma'_{\chi}$ is in the interior of $\Gamma$ and $\Sigma''_{\chi}$ in the exterior. The  Riesz Projector is defined as follows:
\begin{equation}\label{rieszpro}
P_{\Tc}=P_{\Tc,\Gamma}=-\frac{1}{2\pi } \int_{\Gamma} (\Tc-\zeta I)^{-1}d\zeta 
\end{equation}

We recall a few  properties of $P_{\Tc}$, see \cite[III-Thm.~6.17~and~p.~273]{K1976}:
\begin{itemize}
\item[(i)] $P_{\Tc}$ is an orthogonal projector  in $\H_{0,\chi}$.
\item[(ii)] $P_{\Tc}(\H_{0,\chi})$ and $(I-P_{\Tc})(\H_{0,\chi})$ are invariant subspaces for $\Tc$.
\item[(iii)] $\sigma \big({\Tc}_{\rvert_{P_{\Tc}(\H_{0,\chi})}  } \big) =\Sigma'_{\chi}$ and $\sigma \big(\Tc_{\rvert_{(I-P_{\Tc})(\H_{0,\chi})}  } \big) =\Sigma_{\chi} ''$.
\item[(iv)] If  $\Sigma'_\chi=\{\lambda_1, \dots , \lambda_m \}$ where $\lambda_1, \dots , \lambda_m$ are eigenvalues of  $\Tc$ then  $P_{\Tc}(\H_{0,\chi})$ is the  space generated by the corresponding eigenfunctions. 
\end{itemize}


The following theorem is a restatement  of \cite[IV-Thm.~3.16]{K1976} and  \cite[IV-Thm.~3.11, VII-Thm. 1.7]{K1976}. 
For the sake of completeness, we provide some details of the proof. More details can be found in \cite{dalla} where the more general case of symmetrizable operators is considered.

\begin{theo}\label{mainpro} Assume that the family of operators and scalar products are of class  $\C^r$, $r\in \nat\cup\{\infty \}\cup\{\omega\}$. Let $\bar\chi\in\mathcal{X}$ be fixed and $\bar \lambda $ be an eigenvalue of $T_{\bar \chi}$.  Let $\Gamma$ be a curve separating $\Sigma_{{\bar\chi}}'=\{\bar\lambda\}$ from $\Sigma_{{\bar\chi}}''=\sigma (\Ttc) \setminus \{\bar\lambda\}$ as above. Assume that $\bar\lambda $ has multiplicity $m$ and $\bar\lambda=\lambda_{\bar\chi,k}=\dots = \lambda_{\bar\chi, k+m-1}$ for some $k\in \nat$.  Then there exists $\delta >0$ depending only on $T_{\bar \chi}$ and $\Gamma$ such that if $\| \chi-\bar\chi \|<\delta$ then $\Gamma$ separates also the spectrum of $\Tc$ into two parts: $\Sigma_{\chi}'$ in the interior,   $\Sigma_{\chi}''$ in the exterior of $\Gamma$.  Moreover,  $\Sigma_{\chi}'=\{\lambda_{\chi, k}, \dots , \lambda_{\chi, k+m-1}   \}$, the space generated by the eigenfunctions associated with  $\lambda_{\chi, k}, \dots , \lambda_{\chi, k+m-1}$  has dimension $m$   and coincides with $P_{\Tc,\Gamma}[\H_{0,\chi}]$. Finally, the map defined from $\{ \chi\in \mathcal{X}:\ \| \chi-\bar\chi\|<\delta  \}$ to ${\mathcal{L}}(\H_0, \H_0)$ which takes  $\chi$ to $P_{\Tc ,\Gamma}$  is of class  $\C^r$.
\end{theo}

\begin{proof}  By \cite[IV-Thm.~3.16]{K1976}  there exists $\bar \delta >0$ depending only on $\Ttc$ and $\Gamma$ such that if $\hat\delta (\Tc,\Ttc)<\bar \delta $ then  $\Gamma$ separates the spectrum of $\Tc$  in two parts: $\Sigma_{\chi}'$ in the interior and  $\Sigma_{\chi}''$ in the exterior of $\Gamma$. Moreover, the space $P_{\Tc}[\H_{0,\chi}]$ is isomorphic to $P_{\Ttc}[\H_{0,\bar\chi}]$ , hence has dimension $m$.  Since $P_{\Tc}[\H_{0,\chi}]$ is an invariant space for $\Tc$ and $\Tc$ is selfadjoint, we have that the restriction of $\Tc$ to $P_{\Tc}[\H_{0,\chi}]$  is diagonalisable, hence $P_{\Tc}[\H_{0,\chi}]$  has a basis of eigenfunctions of $\Tc$ corresponding to $m$ possibly multliple eigenvalues.  Moreover, the same argument applied to all eigenvalues of $\Ttc$ not exceeding $\bar \lambda$  allows to conclude that  by taking $\delta$ sufficiently small we have that  $\Sigma_{\chi}'=\{\lambda_{\chi, k}, \dots , \lambda_{\chi, k+m-1}   \}$.  Thus, in order to complete the proof of the first part of the statement it suffices to note that  by Lemma~\ref{gapo1} it follows that there exists $\delta>0$ such that if $\| \chi-\bar\chi\|<\delta $ then 
$\hat\delta (\Tc,\Ttc)<\bar \delta $.

The continuous dependence of $P_{\Tc}$ on $\chi$ is deduced from formula \eqref{rieszpro}  and is proved in \cite[IV-Thm.~3.16]{K1976}. In order to prove 
 the dependence o $P_{\Tc}$ on $\chi$ is of class $\C^r$, we use formula \eqref{res1}   and note that the two maps 
 $ \chi \mapsto \left(\TAc-\zeta \Jc   \right)$  and     $ \chi\mapsto \Jc   $ are of class $\C^r$ from $\mathcal{X}$ to the space   
  Banach space $\C(\Gamma , {\mathcal{L}}(V,V'_{\sharp})   )$ of continuous functions from $\Gamma $ to ${\mathcal{L}}(V,V'_{\sharp}) $ endowed with the natural sup-norm defined by  
$
\| f \|_{ \C(\Gamma , {\mathcal{L}}(V,V'_{\sharp})    ) }=\sup_{\zeta \in \Gamma }\| f(\zeta )  \|_{{\mathcal{L}}(V,V'_{\sharp})    }  \, .
$
Since the inversion operator is real-analytic it follows that the map $ \chi \mapsto \left(\TAc-\zeta \Jc   \right)^{-1}\circ \Jc$ is of class $\C^r$ in a neighborhood of $\bar\chi$. In order to conclude it suffices to observe that integration on $\Gamma$ defines a linear and continuous map, hence analytic map.  
\end{proof}

Under  the same assumptions of Theorem~\ref{mainpro} we set $F=\{k, \dots, k+m-1\}$ and  we consider the  elementary symmetric functions of the corresponding eigenvalues
\begin{equation}\label{symfunct}
\Lambda_{F,s}[\chi] := \sum_{\substack{j_1,\dots, j_s \in F\\j_1<\dots<j_s}} \lambda_{j_1,\chi} \cdot \cdot \cdot \lambda_{j_s,\chi}, \qquad s=1,\dots, m.
\end{equation}

We plan to prove that these functions are of class $\C^r$. To do so, we follow the argument in \cite{LLdC2004a}.  We fix  an orthonormal basis  $\bar u_1,\dots \bar u_m$ in $\H_{0,\bar\chi}$ of the eigenspace associated with
$\bar \lambda$ and we note that, possibly shrinking $\delta$, the vectors $P_{\Tc ,\Gamma}\bar u_1, \dots , P_{\Tc ,\Gamma}\bar u_m$ are linearly independent. 
Thus, it is possible to apply the Gram-Schmidt procedure to orthonormalize those vectors  in $\H_{0,\chi}$ and define a basis  of orthonormal vectors 
$
u_{\chi, 1}, \dots, u_{\chi,m}
$
in $\H_{0,\chi}$  for the space $P_{\Tc ,\Gamma}[\H_{0,\chi}]$. Since the Gram-Schmidt procedure involves elementary operations and the scalar products $\scp{u}{v}_{\H_{0,\chi}}$ with $v \in \V$ are of class $\C^r$ with respect to $\chi$ by our assumptions, it follows that $
u_{\chi, 1}, \dots, u_{\chi,m}
$ are of class $\C^r$  from a neighborhood of $\bar \chi $ in $\mathcal X$ to $\H_{0}$. 

In order to reduce our problem to finite dimension, it would be natural to consider the restriction of the operator $\Tc$ to the $m$-dimensional space 
$P_{\Tc ,\Gamma}[\H_{0,\chi}]$ and to represent it by a matrix with respect to the orthonormal basis defined above. However, since the domain of $\Tc$ depends on $\chi$, it is not straightforward to prove that the entries of that matrix are of class $\C^r$ with respect to $\chi$. For this reason, it is safer to use the resolvent of $\Tc$ rather than $\Tc$ itself. Thus we consider the resolvent operator for $\zeta =-1$ and we set 
$$
R_{\chi}= (\Tc +I)^{-1}\, .
$$ 


We denote by $\mathcal{S}=(\mathcal{S}_{hl})_{h,l=1,\dots m}$ the $m\times m$ symmetric matrix with respect to the basis $u_{\chi, 1}, \dots, u_{\chi,m}$ of the restriction of $R_{\chi}$ to the $m$-dimensional space 
$P_{\Tc ,\Gamma}[\H_{0,\chi}]$.  Clearly, we have 
\begin{equation}
{\mathcal{S}}_{hl}=\scp{R_{\chi} u_{   \chi, h}  }{u_{\chi , l}   }_{\H_{0,\chi}}  = \scp{\left(\TAc+ \Jc   \right)^{-1}\circ \Jc u_{   \chi, h}}{u_{\chi , l}}_{\H_{0,\chi}}.            \end{equation}

Recall that by Theorem~\ref{mainpro} the eigenvalues of the matrix $\mathcal{S}$ are given by  $(\lambda_{\chi, k}+1)^{-1}, \dots ,$ $ ( \lambda_{\chi, k+m-1}+1)^{-1}$.

Then it  is natural to consider  the functions 
\begin{equation}
    \hat{\Lambda}_{F,s}[\chi]:=\sum_{\substack{j_1,\dots, j_s \in F\\j_1<\dots<j_s}} (\lambda_{j_1,\chi}+1)\cdot \cdot \cdot (\lambda_{j_s,\chi}+1),
\end{equation}
 and to note that 
\begin{equation} \label{lambdatilde}
    \Lambda_{F,s}[\chi] = \sum_{p=0}^s (-1)^{s-p} \binom{|F|-k}{s-p} \hat{\Lambda}_{F,p}[\chi],
\end{equation}
where we have set $\Lambda_{F,0}= \hat{\Lambda}_{F,0}=1$.  Thus,  in order to study the dependence of the  functions $ \Lambda_{F,s}[\chi]$ on $\chi$, it is enough to consider the functions $\hat{\Lambda}_{F,s}[\chi]$. To do so, it is convenient  to set
\begin{equation}\label{lambdatildeemme}
    \widehat{M}_{F,s}[\chi]:=\sum_{\substack{j_1,\dots, j_s \in F\\j_1<\dots<j_s}} (\lambda_{j_1,\chi}+1)^{-1}\cdot \cdot \cdot (\lambda_{j_s,\chi}+1)^{-1},
\end{equation}
and to note that
\begin{equation} \label{lambdatildeemmecappuccio}
    \hat{\Lambda}_{F,s}[\chi]=\frac{ \widehat{M}_{F,m-s}[\chi]}{ \widehat{M}_{F,m}[\chi]} .
\end{equation}

We are now ready to prove the following

\begin{theo}\label{symtheo} Let the same assumptions of Theorem~\ref{mainpro}  hold. Then there exists an open neighborhood $ {\mathcal{X}}_{\bar \chi}$ of $\bar \chi$ in $\mathcal{X}$ such that the restriction of the functions   $\Lambda_{F,s}$ to $ {\mathcal{X}}_{\bar \chi}$ are of class $\C^r$ for all $s=1,\dots, m$. Moreover, the directional derivatives of the symmetric functions $\Lambda_{F,s}[\chi]$  at the point $\bar \chi$ in the direction $\dot\chi$  are given by the  formula
\begin{equation} \label{matrixrellichtrace}
\partial_{\dot\chi}  \Lambda_{F,s}[\chi]_{\big\rvert_{\chi=\bar\chi}} = \bar \lambda ^{s-1}\binom{m-1}{s-1}\sum_{h=1}^m \left(
 \partial_{\dot\chi} 
\scp{\A_{\chi} u_{  \bar  \chi, h}}{\A_{\chi} u_{  \bar  \chi, h}}_{\H_{1,\chi}}
-  \bar \lambda  \partial_{\dot\chi}    \scp{ u_{  \bar  \chi, h}}{ u_{  \bar  \chi, h}}_{\H_{0,\chi}}
\right)
\end{equation}
for $s=1,\dots, m$, where $u_{  \bar  \chi, 1},\dots, u_{  \bar  \chi, m}$ is an orthonormal basis in $\H_{0,\bar\chi}$ of eigenvectors associated with
$\bar \lambda$.

\end{theo}

\begin{proof} Up to the sign, the functions $\widehat M_{F,s}[\chi]$ are the coefficients of the characteristic polynomial ${\rm det} ({\mathcal{S}}_{hk}-\lambda I  ) $ of the matrix $
{\mathcal{S}}_{hk}$. Since those coefficients are obtained as sums of products of the coefficients of $
{\mathcal{S}}_{hk}$ which are functions of class $\C^r$ in the variable $\chi$, for $\chi$ in a suitable neghborhood $ {\mathcal{X}}_{\bar\chi}$ of $\bar\chi$, it follows they are also of class $\C^r$ on  $ {\mathcal{X}}_{\bar\chi}$. By formulas \eqref{lambdatilde} and   \eqref{lambdatildeemmecappuccio} we deduce the same results for the functions $\Lambda_{F,s}[\chi]$ and the proof of the first part of the theorem is complete. In order to prove formula \eqref{matrixrellichtrace} it suffices to differentiate in formula \eqref{symfunct}.  This requires some care. Namely, we fix a direction $\dot\chi\in X$ and we consider the one parameter family $\mathbb{R} \ni \tau\mapsto \chi_{\tau}:=\bar \chi +\tau \dot\chi$. By  Theorem~\ref{Rellichthm} the eigenvalues 
$\lambda_{j,\chi_{\tau}}$ with $j\in F$ admit right (as well as left) derivatives  ${\frac{d  \lambda_{j,\chi_{\tau}}}{d\tau}}_{\vert_{\tau=0^+}}$. 
Thus by differentiating formula  \eqref{symfunct} we get 
\begin{equation}
\partial_{\dot\chi}  \Lambda_{F,s}[\chi]_{\big\rvert_{\chi=\bar\chi}}=   
  \bar\lambda^{s-1}\sum_{\substack{j_1,\dots, j_{s} \in F\\j_1<\dots<j_{s}}} {\frac{d  \lambda_{j_1,\chi_{\tau}}}{d\tau}}_{\vert_{\tau=0^+}}
 \!\!\!\! +\dots +{\frac{d  \lambda_{j_s,\chi_{\tau}}}{d\tau}}_{\vert_{\tau=0^+}}=\bar\lambda^{s-1}\binom{m-1}{s-1} \sum_{j\in F}  {\frac{d  \lambda_{j,\chi_{\tau}}}{d\tau}}_{\vert_{\tau=0^+}}.
\end{equation}
Finally, it suffices to observe that by Theorem~\ref{Rellichthm}  the sum in the last term  of the previous formula  is the trace of the matrix \eqref{matrixrellich}, that is  $\sum_{h=1}^m \left(
 \partial_{\dot\chi} 
\scp{\A_{\chi} u_{  \bar  \chi, h}}{\A_{\chi} u_{  \bar  \chi, h}}_{\H_{1,\chi}}
-  \bar \lambda  \partial_{\dot\chi}    \scp{ u_{  \bar  \chi, h}}{ u_{  \bar  \chi, h}}_{\H_{0,\chi}}
\right)
$.
\end{proof}

The following bifurcation result  is deduced by the classical Rellich theorem for symmetric matrixes which, up to ur knowledge, is know for perturbations of class $\C^1$ or analytic. Note that it is necessary to restrict to families of operators depending on one scalar parameter since counterexamples are well known for families depending on two or more parameters.

\begin{theo}[Rellich] \label{Rellichthm} Let the same assumptions of Theorem~\ref{mainpro} hold with $X={\mathbb{R}}$ and $r=1,\omega$. Then for $\chi$ sufficiently close to $\bar \chi$, the eigenvalues  $\lambda_{\chi, k}, \dots ,$ $  \lambda_{\chi, k+m-1}$ can be represented by $m$      functions $\chi \mapsto \zeta_{\chi, k}, \dots , \zeta_{\chi, k+m-1}$ of class $\C^r$, defined in a suitable neighborhood of $\bar \chi$.  Moreover, given an orthonormal basis of eigenvectors $u_{\bar\chi, 1}, \dots, u_{\bar\chi,m} $ in $\H_{0,\bar\chi}$   for the eigenspace of $\bar\lambda$,  the derivatives in the variable $\chi\in {\mathbb{R}}$ of those functions at the point $\bar \chi$  are given by the eigenvalues of the matrix   
\begin{equation}\label{matrixrellich}\left(  {\frac{d}{d\chi}}\biggl( \scp{\A_{\chi} u_{  \bar  \chi, h}}{\A_{\chi} u_{  \bar  \chi, l}}_{\H_{1,\chi}}\biggr)_{\big\vert_{\chi=\bar\chi}}
-  \bar \lambda  {\frac{d}{d\chi}}\biggl( \scp{ u_{  \bar  \chi, h}}{ u_{  \bar  \chi, l}}_{\H_{0,\chi}}\biggr)_{\big\vert_{\chi=\bar\chi}}
 \right)_{h,l=1,\dots , m}.
\end{equation}
\end{theo}  

\begin{proof}  
Since the parameter $\chi$ is assumed to be real and the entries of the matrix  ${\mathcal{S}}$ are functions  analytic in $\chi$ or of class $\C^1$ then the classical Rellich Theorem  is applicable  (see \cite[Thm. 1, p.33]{Re1969} and  to \cite[Ch. II, Thms. 6.1, 6.8]{K1976})   and it allows to represent the eigenvalues of ${\mathcal{S}}$   by means of $m$  functions $\chi \mapsto \hat \zeta_{\chi, k}, \dots , \hat \zeta_{\chi, k+m-1}$ which are analytic or of class $C^1$, depending on the regularity assumptions. 
 Recall that the numbers  $(\lambda_{\chi, k}+1)^{-1}, \dots ,$ $ ( \lambda_{\chi, k+m-1} +1)^{-1}$ are exactly the eigenvalues of the $m\times m$ symmetric matrix ${\mathcal{S}}$
defined above. Then, by setting $\zeta_{\chi, k}=  \hat \zeta_{\chi _k}^{-1}-1$   we conclude the proof of the first part of the statement.

By \cite[Ch. II, Thm. 5.4]{K1976} the derivatives  of $\hat \zeta_{\chi, k}, \dots , \hat \zeta_{\chi, k+m-1}$  with respect to $\chi$ at $\bar \chi$  are given by the eigenvalues of the matrix $\frac{d}{d\chi}{\mathcal{S}}$, which we now calculate. First of all 
we note that by differentiating ${\mathcal{S}}_{hl}$ with respect to $\chi $ at the point $\bar \chi$  we get
\begin{eqnarray}\label{dermat}\lefteqn{
\frac{d}{d\chi}{\mathcal{S}}_{hl}=   \scp{\frac{d}{d\chi}R_{\chi}  u_{  \bar  \chi, h}  }{   u_{\bar \chi , l}   }_{\H_{0,\bar \chi}}+
\frac{d}{d\chi}   \scp{R_{\bar\chi} u_{  \bar \chi, h}  }{u_{\bar\chi , l}   }_{\H_{0,\chi}}  }\nonumber \\ 
& &\qquad\qquad +
  \scp{R_{\bar\chi}   \frac{d}{d\chi}  u_{   \chi, h}  }{u_{\bar\chi , l}   }_{\H_{0,\bar \chi}}+  \scp{R_{\bar\chi}   u_{  \bar  \chi, h}  }{  \frac{d}{d\chi}  u_{\chi , l}   }_{\H_{0,\bar \chi}}  = \scp{\frac{d}{d\chi}R_{\chi}   u_{  \bar  \chi, h}  }{   u_{\bar \chi , l}   }_{\H_{0,\bar \chi}}
\end{eqnarray}
where we have used the fact that 
\begin{eqnarray}\frac{d}{d\chi}   \scp{R_{\bar\chi} u_{  \bar \chi, h}  }{u_{\bar\chi , l}   }_{\H_{0,\chi}} +
 \scp{R_{\bar\chi}  \frac{d}{d\chi}  u_{   \chi, h}  }{u_{\bar\chi , l}   }_{\H_{0,\bar \chi}}+  \scp{R_{\bar\chi}   u_{  \bar  \chi, h}  }{  \frac{d}{d\chi}  u_{\chi , l}   }_{\H_{0,\bar \chi}} \nonumber \\
 = (\bar \lambda +1)^{-1} (
\frac{d}{d\chi}   \scp{ u_{  \bar \chi, h}  }{u_{\bar\chi , l}   }_{\H_{0,\chi}} +
 \scp{  \frac{d}{d\chi}  u_{   \chi, h}  }{u_{\bar\chi , l}   }_{\H_{0,\bar \chi}}+  \scp{   u_{  \bar  \chi, h}  }{  \frac{d}{d\chi}  u_{\chi , l}   }_{\H_{0,\bar \chi}})  \nonumber  \\
 =
 (\bar \lambda +1)^{-1}  \frac{d}{d\chi}   (  \scp{ u_{   \chi, h}  }{u_{\chi , l}   }_{\H_{0,\chi}} )=  (\bar \lambda +1)^{-1} \lambda \frac{d}{d\chi} \delta_{hl}=0.
 \end{eqnarray}

  Using formula \eqref{dermat} we have
\begin{eqnarray} \lefteqn{ \frac{d}{d\chi}{\mathcal{S}}_{hl} =  \scp{\frac{d}{d\chi}R_{\chi}   u_{  \bar  \chi, h}  }{   u_{\bar \chi , l}   }_{\H_{0,\bar \chi}}=
 \scp{\frac{d}{d\chi}  \left(\TAc+ \Jc   \right)^{-1}\circ \Jc u_{  \bar  \chi, h}  }{   u_{\bar \chi , l}   }_{\H_{0,\bar \chi}} } \nonumber \\
 & & \quad
=-
 \scp{ \left(\TAtc+ \Jtc   \right)^{-1} \circ \frac{d}{d\chi}    \left(\TAc+ \Jc   \right)  \circ    \left(\TAtc+ \Jtc   \right)^{-1}  \circ \Jtc u_{  \bar  \chi, h}  }{   u_{\bar \chi , l}   }_{\H_{0,\bar \chi}}\nonumber   \\
& &\quad  +   \scp{  \left(\TAtc+ \Jtc   \right)^{-1}\circ   \frac{d}{d\chi} \Jc u_{  \bar  \chi, h}  }{   u_{\bar \chi , l}   }_{\H_{0,\bar \chi}} \nonumber  \\
& &\quad  = -(\bar \lambda +1)^{-2}
  \scp{ \frac{d}{d\chi}    \left(\TAc+ \Jc   \right)  u_{  \bar  \chi, h}  }{   u_{\bar \chi , l}   }   +  (\bar \lambda +1)^{-1} \scp{   \frac{d}{d\chi} \Jc u_{  \bar  \chi, h}  }{   u_{\bar \chi , l}   }.
\end{eqnarray}
Then the derivatives of the functions $ \zeta_{\chi, k}, \dots ,  \zeta_{\chi, k+m-1}$ are given by the eigenvalues of the matrix  $\frac{d}{d\chi}{\mathcal{S}}$ multiplied by $-(\bar \lambda +1)^2$ i.e., is the matrix with the following entries
\begin{eqnarray}
  \scp{ \frac{d}{d\chi}    \left(\TAc+ \Jc   \right)  u_{  \bar  \chi, h}  }{   u_{\bar \chi , l}   }  -  (\bar \lambda +1)\scp{   \frac{d}{d\chi} \Jc u_{  \bar  \chi, h}  }{   u_{\bar \chi , l}   }  \nonumber \\
  =  \scp{ \frac{d}{d\chi}    \TAc \ u_{  \bar  \chi, h}  }{   u_{\bar \chi , l}   } -  \bar \lambda \scp{  \frac{d}{d\chi}    \Jc u_{  \bar  \chi, h}  }{   u_{\bar \chi , l}   }
\end{eqnarray}
which coincides with those in the statement.
\end{proof}

\begin{rem}[\bf  Hellmann-Feymann Theorem] \label{rulethumb}
   Recall that the eigenvalues can be represented by means of the Rayleigh quotient as  follows
\begin{equation}
\label{ray}\lambda_{\chi ,k}=\frac{   \scp{\A_{\chi} u_{    \chi, k}}{\A_{\chi} u_{   \chi, k}}_{\H_{1,\chi}}
}{   \scp{ u_{   \chi, k}}{ u_{   \chi, k}}_{\H_{0,\chi}}   }  .
\end{equation} 
Assume that $\bar\lambda=\lambda_{\bar \chi ,k}$ is a simple eigenvalue.  Then  $\lambda_{ \chi ,k}$ is also simple for $\chi $ close to $\bar \chi$ and the derivative of  $\lambda_{ \chi ,k}$ at the point $\bar\chi$ is obtained by differentiating the quotient
\begin{equation}\label{rule1}\frac{   \scp{\A_{\chi} u_{  \bar  \chi, k}}{\A_{\chi} u_{  \bar  \chi, k}}_{\H_{1,\chi}}
}{   \scp{ u_{  \bar  \chi, k}}{ u_{  \bar  \chi, k}}_{\H_{0,\chi}}   }
\end{equation} 
with respect to $\chi$ at the point $\bar \chi$.  If $u_{  \bar\chi, k}$ is normalized by $  \scp{ u_{\bar\chi, k}  }{u_{\bar\chi , k}   }_{\H_{0,\bar \chi}}=1$ then such derivative in the direction $\dot\chi$  at the point $\bar \chi$ is given by 
\begin{equation}\label{rule2}
{\partial_{\dot\chi}\lambda_{\chi ,k}}_{\big\rvert_{\chi=\bar\chi}}=  \partial_{\dot\chi}   \scp{\A_{\chi} u_{  \bar  \chi, k}}{\A_{\chi} u_{  \bar  \chi, k}}_{\H_{1,\chi}} -  \bar \lambda  \partial_{\dot\chi}  \scp{ u_{  \bar  \chi, k}}{ u_{  \bar  \chi, k}}_{\H_{0,\chi}}  ,
\end{equation}
see formula \eqref{matrixrellichtrace}.
For multiple eigenvalues and $\chi$ real,  the derivatives of the branches of eigenvalues splitting $\bar \lambda $ are given by the matrix obtained in the same way, as in formula 
\eqref{matrixrellich}.

We note that formula \eqref{rule2} can be obtained in a direct way by assuming apriori that a normalized eigenfunction
$u_{    \chi, k}$ corresponding to $\lambda_{ \chi ,k}$  is differentiable with respect to $\chi$ at the point $\bar \chi$. Indeed, by differentiating formula \eqref{ray} and taking into account the normalization of  $u_{   \chi, k}$  we get
\begin{eqnarray*}\lefteqn{
{\partial_{\dot\chi}\lambda_{\chi ,k}}_{\big\rvert_{\chi=\bar\chi}}  =  \partial_{\dot\chi} \scp{\A_{\chi} u_{   \bar \chi, k}}{\A_{\chi} u_{ \bar  \chi, k}}_{\H_{1,\chi}}     +
\scp{\A_{\bar \chi} \partial_{\dot\chi} u_{    \chi, k}}{\A_{\bar\chi} u_{  \bar  \chi, k}}_{\H_{1,\bar\chi}}+\scp{\A_{\bar \chi} u_{   \bar \chi, k}}{\A_{\bar\chi} \partial_{\dot\chi} u_{   \chi, k}}_{\H_{1,\bar\chi}}  } \\ & & 
-\bar\lambda \left(  \partial_{\dot\chi}\scp{ u_{  \bar \chi, k}}{ u_{  \bar \chi, k}}_{\H_{0,\chi}}   +\scp{  \partial_{\dot\chi}u_{   \chi, k}}{ u_{   \chi, k}}_{\H_{0,\bar\chi}}           +\scp{ u_{   \chi, k}}{ \partial_{\dot\chi} u_{   \chi, k}}_{\H_{0,\bar\chi}}   \right)\\
& & 
= \partial_{\dot\chi} \scp{\A_{\chi} u_{   \bar \chi, k}}{\A_{\chi} u_{ \bar  \chi, k}}_{\H_{1,\chi}}   
 -\bar\lambda  \partial_{\dot\chi}\scp{ u_{  \bar \chi, k}}{ u_{  \bar \chi, k}}_{\H_{0,\chi}}  \\
& &  +2\Re\left( \scp{\A_{\bar \chi} u_{   \bar \chi, k}}{\A_{\bar\chi} \partial_{\dot\chi} u_{   \chi, k}}_{\H_{1,\bar\chi}}   -\bar\lambda\scp{ u_{   \chi, k}}{ \partial_{\dot\chi} u_{   \chi, k}}_{\H_{0,\bar\chi}}   \right)
 = \partial_{\dot\chi} \scp{\A_{\chi} u_{   \bar \chi, k}}{\A_{\chi} u_{ \bar  \chi, k}}_{\H_{1,\chi}}   \\ & & 
 -\bar\lambda  \partial_{\dot\chi}\scp{ u_{  \bar \chi, k}}{ u_{  \bar \chi, k}}_{\H_{0,\chi}}  
 +2\Re \scp{\A_{\bar\chi} ^* \A_{\bar \chi} u_{   \bar \chi, k} -\bar\lambda   u_{   \bar \chi, k} }{ \partial_{\dot\chi} u_{   \chi, k}}_{\H_{1,\bar\chi}}    \\
 & & 
 = \partial_{\dot\chi} \scp{\A_{\chi} u_{   \bar \chi, k}}{\A_{\chi} u_{ \bar  \chi, k}}_{\H_{1,\chi}}   
 -\bar\lambda  \partial_{\dot\chi}\scp{ u_{  \bar \chi, k}}{ u_{  \bar \chi, k}}_{\H_{0,\chi}}  
\end{eqnarray*} 
where we have used the fact that $  \A_{\bar\chi} ^* \A_{\bar \chi} u_{   \bar \chi, k} =\bar\lambda   u_{   \bar \chi, k} $. This formal procedure was used in the proof of our Theorem~4.1 in Part I of this series of papers and it is essentially the same used in the classical Hellmann-Feynman theorem. We refer to  \cite{este} for a recent discussion on this topic and references. 
\end{rem}

\section{Applications to the De Rham Complex}
\label{applications}

In this section we apply the abstract results obtained in Section \ref{perturbtheory} to the Maxwell and Helmholtz problems \eqref{maxprobl}, \eqref{helmprobl}. To do so, we shall make the following assumptions:

\begin{itemize}
\item the space of the parameters $\chi$  is a generic  real Banach space $X$ and in particular $\chi$ belongs to an open subset $\mathcal{X}$ of $X$; in the case of the Rellich-type Theorems \ref{theo:shapederivatives},  \ref{theo:shapederivativeshe} we shall consider the particular case $X=\mathbb{R}$;
\item we consider a family of  bi-Lipschitz transformations $\{\Pc\}_{\chi \in \mathcal{X}}\subset  \mathcal{L}(\Omega) $ where $  \mathcal{L}(\Omega)$ is as in the Introduction,  and we assume that the map $\chi \mapsto \Phi_{\chi}$
is of class $\C^r$ with $r\in \nat \cup\{\infty\}\cup\{\omega\}$;
\item we shall use the following notation:
\begin{equation}\label{defpsi}
 \wt\Psi :=  {\partial_{\dot\chi}\Phi_{\chi}}_{\big\rvert_{\chi=\bar\chi}}  \qquad \text{and} \qquad \Psi :=  \wt\Psi\circ \Phi_{\bar\chi}^{(-1)}
\end{equation}
to denote the directional derivative of $\Phi_{\chi}$ at the point $\bar \chi$ in the direction $\dot\chi$ and its push-forward via $\Phi_{\bar\chi}$, respectively.  (Note that if $X=\mathbb{R}$ then we shall talk about derivative rather than directional derivative.) 
\item we assume in this section that   $\eps,\mu$, as well as to the scalar function $\nu$, are defined in the whole of $\mathbb{R}^3$  and they are at least of class $\C^1$.  Moreover, we assume that the (real valued symmetric) matrices $\eps (x),\mu (x)$ are positive-definite for all $x\in \mathbb{R}^3$ and that $\nu $ is also positive. 
\end{itemize}

Later we shall specify the Hilbert spaces associated with any $\chi$,  depending on the problem under consideration. Before doing so 
we need some technical lemmas.

\begin{lem}\label{epsreg}
Assume that $M: \mathbb{R}^3 \to \mathbb{R}^{3 \times 3}$ is a map  of class $\C^r$ with $r\in \nat \cup\{\infty\}\cup\{\omega\}$. 
Consider the map
$$\mathcal{M}: \C(\overline{\Omega}; \mathbb{R}^3) \to \C(\overline{\Omega}; \mathbb{R}^{3 \times 3})$$
$$\Phi \mapsto M \circ \Phi.$$
Then $\mathcal{M}$ is  of class $\C^r$  and its Fr\'{e}chet differential $d_\Phi \mathcal{M}$ at a point  $\Phi \in \C(\overline{\Omega}; \mathbb{R}^3)$, applied to   $\Psi \in \C(\overline{\Omega}; \mathbb{R}^3)$, is given by the formula
$$d_\Phi \mathcal{M} (\Psi) = \left( (\nabla M_{i,j} \circ \Phi) \cdot \Psi \right)_{1\leq i,j \leq 3}.$$
\end{lem}

\begin{proof}   This statement can be deduced as a limiting case of Theorem 2.10 and Proposition 2.17 in \cite{LA1998} (in the notation of  \cite{LA1998} one should set $m=\alpha=\beta=0$) the proof being obtained in the same way as a simpler application of the abstract results in \cite{LA2000} (for that purpose note that if $M$ is analytic then its restriction to any compact set belongs in the Roumieu class as required in Proposition 2.17 in \cite{LA1998}, see e.g., \cite[Thm.~2.17]{DALAMU2021}). We also refer to \cite{valent} for a direct approach to differentiability and analyticity results for the composition operator in Sobolev and Schauder spaces.   Note that $\C(\overline{\Omega}; \mathbb{R})$ is a Banach algebra (an important property for the proof of these results) regardless of the boundary regularity of $\Omega$. \end{proof}

Recall that $\tilde \eps= \eps \circ \Phi_\chi$.
We apply Lemma \ref{epsreg} to the maps  $\eps,\mu$ as well as to the scalar function $\nu$ (recall that we assume that they are of class $\C^1$). In particular it follows that 
\begin{equation}
{\partial_{\dot\chi} \wt \eps}_{\rvert_{\chi=\bar\chi}} = \left( (\nabla \eps_{i,j} \circ \Phi_{\bar\chi}) \cdot \wt \Psi \right)_{1\leq i,j, \leq 3}.
\end{equation}
Then we set 
\begin{equation}\label{dzeta}
\partial_\Psi \eps:=\bigl( {\partial_{\dot\chi} \wt \eps}_{\rvert_{\chi=\bar\chi}}  \bigr) \circ \Phi_{\bar\chi}^{-1} = \left( \nabla \eps_{i,j} \cdot \Psi \right)_{1\leq i,j, \leq 3}.
\end{equation}
Similarly, we set 
\begin{equation}\label{dzetamu}
\partial_\Psi \mu^{-1}:=\bigl( {\partial_{\dot\chi} \wt \mu^{-1}}_{\rvert_{\chi=\bar\chi}} \bigr) \circ \Phi_{\bar\chi}^{-1} = \left( \nabla \mu^{-1}_{i,j} \cdot \Psi \right)_{1\leq i,j, \leq 3}
\end{equation}
and 
\begin{equation*}
\partial_\Psi \nu :=  \bigl( {\partial_{\dot\chi} \wt \nu}_{\rvert_{\chi=\bar\chi}}  \bigr)  \circ \Phi_{\bar\chi}^{-1} =  \nabla \nu \cdot \Psi,
\end{equation*}
where $\wt\Psi$ and $\Psi$ are defined in \eqref{defpsi}.

Recall from the introduction the following definitions:
\begin{align*}
\eps_{\Phi}
=\tau^{2}_{\Phi}\eps\tau^{1}_{\Phi^{-1}}
&=(\det J_{\Phi})J_{\Phi}^{-1}\wt{\eps}J_{\Phi}^{-\top},
&
\nu_{\Phi}
=\tau^{3}_{\Phi}\nu\tau^{0}_{\Phi^{-1}}
&=(\det J_{\Phi})\wt{\nu},\\
\mu_{\Phi}^{-1}
=\tau^{1}_{\Phi}\mu^{-1}\tau^{2}_{\Phi^{-1}}
&=(\det J_{\Phi})^{-1}J_{\Phi}^{\top}\wt{\mu^{-1}}J_{\Phi}.
\end{align*}

\begin{lem} \label{lemmafrechet}   Assume that the maps $\nu:\mathbb{R} \to \mathbb{R}$,  $\eps , \mu : \mathbb{R}^3 \to \mathbb{R}^{3 \times 3}$ as well as
 the map $\chi \mapsto \Phi_{\chi}$
are  of  class $\C^r$ with $r\in \nat \cup\{\infty\} \cup\{\omega\}$.  Then  the maps
$
\chi \mapsto J_{\Phi_\chi},\  J_{\Phi_\chi}^{-1}, \eps_{\Phi_\chi},\  \mu^{-1}_{\Phi_\chi}
$
from $\mathcal{X}$ to  $\L^{\infty}(\Omega; \mathbb{R}^{3\times 3})$ and the maps 
 $\chi \mapsto \operatorname{det} J_{\Pc}, \ (\operatorname{det} J_{\Pc})^{-1}, \nu_{\Phi_\chi}$ from $\mathcal{X}$ to   $\L^{\infty}(\Omega; \mathbb{R})$
 are of class $\C^r$.  
\end{lem} 
\begin{proof} It is well-known that the determinant and the inverse of a matrix define real analytic operators. Moreover, the Jacobian 
operator  from $\C^{0,1}(\overline{\Omega}; \mathbb{R}^3)$  to  $\L^{\infty}(\Omega; \mathbb{R}^{3\times 3})$ which takes  $\Phi$ to $J_\Phi$ is a linear and continuous map, hence real-analytic.  It follows by our assumptions on $\nu,\eps,\mu$  and by Lemma~\ref{epsreg} that  the maps in the statement  are of class $\C^r$ in the variable $\chi$ since they are compositions of analytic maps and  maps of class $\C^r$.
\end{proof}

\subsection{Maxwell Eigenvalues} \label{maxshapeder}

In this subsection we consider the Maxwell problem \eqref{maxprobl}.   As mentioned in the introduction, we shall prove a formula for the shape derivatives of its eigenvalues  under assumptions on $\Omega$ and $\Phi$ which are weaker than those used  in \cite{LZ2021a}. Note that  in \cite[Thm. 4.5]{LZ2021a} the authors require additional hypothesis on the regularity of the eigenvectors, assuming them of class $\H^2$. 
Here we only assume the validity of the Gaffney inequality on the reference domain $\Omega_\Phi$ around which we differentiate.  This  inequality   for  general mixed boundary conditions  on a domain $\Omega_\Phi$ reads as follows. There exists a $C>0$ such that 
\begin{subequations}
\label{gaffneytoomuch}
\begin{equation}
\label{gaffneytoomuchelec}
\norm{u}_{\H^1(\Omega_\Phi)} \leq C \left( \norm{u}_{\L^2(\Omega_\Phi)} + \norm{\rot u}_{\L^2(\Omega_\Phi)} + \norm{\div \eps u}_{\L^2(\Omega_\Phi)} \right)
\end{equation}
for all $u \in \R_{\gatp}(\Omega_\Phi) \cap \eps^{-1} \D_{\ganp} (\Omega_\Phi)$, and
\begin{equation}
\label{gaffneytoomuchmagn}
\norm{u}_{\H^1(\Omega_\Phi)} \leq C \left(\norm{u}_{\L^2(\Omega_\Phi)} + \norm{\rot \mu^{-1} u}_{\L^2(\Omega_\Phi)} + \norm{\div  u}_{\L^2(\Omega_\Phi)} \right)
\end{equation}
for all $u \in \mu\,\R_{\ganp}(\Omega_\Phi) \cap \D_{\gatp} (\Omega_\Phi)$.
\end{subequations}

\begin{rem}   Assume that  $\Gamma_t$ and $\Gamma_n$ are separated (in other words $\Gamma=\Gamma_t\cup\Gamma_n$ and 
$\ol{\ga}_{t}\cap\ol{\ga}_{n}=\emptyset$). 
Then if    $\Omega_\Phi$ is a bounded domain of  class $\C^{1,1}$ or a bounded  convex domain, and the coefficients $\eps,\mu$ are of class $\C^1$,  the Gaffney inequalities \eqref{gaffneytoomuchelec}, \eqref{gaffneytoomuchmagn} hold on  $\Omega_\Phi$, see \cite{filonov}. We note that assuming that  $\gatp$ and $\ganp$ can be swapped then  \eqref{gaffneytoomuchmagn} follows from \eqref{gaffneytoomuchelec}.
\end{rem}

In view of  the notation and the results of  Section~\ref{sec:shape2sound}, we need  to  specialise the function spaces and the operators as follows:
\begin{itemize}
\item  we fix the ambient Hilbert spaces  by setting $\H_0 = \H_1= \L^2(\Omega)$ and we consider also the 
Hilbert spaces  $\H_{0,\chi} = \L^{2}_{\eps_{{\Phi}_{\chi}}}(\om)$  and $\H_{1,\chi} = \L^{2}_{\mu_{{\Phi}_{\chi}}}(\om)$  depending on $\chi$. Note that changing the value of $\chi$ gives equivalent scalar products;
\item  we consider the operator 
$$\A_\chi:=\mu_{\Phi_{\chi}}^{-1}\rot_{\gat}:\R_{\gat}(\om)\subset\L^{2}_{\eps_{\Phi_{\chi}}}(\om)\to\L^{2}_{\mu_{\Phi_{\chi}}}(\om)$$
as the pull-back of the operator
$$\A:=\mu^{-1}\rot_{\gatpchi}:\R_{\gatpchi}(\Omega_{\Phi_{\chi}})\subset\L^{2}_{\eps}(\Omega_{\Phi_{\chi}})\to\L^{2}_{\mu}(\Omega_{\Phi_{\chi}});$$
recall  their adjoints
$$\A_\chi^{*}=\eps_{\Phi_{\chi}}^{-1}\rot_{\gan}:\R_{\gan}(\om)\subset\L^{2}_{\mu_{\Phi_{\chi}}}(\om) \to\L^{2}_{\eps_{\Phi_{\chi}}}(\om),$$
$$\A^{*}=\eps^{-1}\rot_{\ganpchi}:\R_{\ganpchi}(\Omega_{\Phi_{\chi}})\subset\L^{2}_{\mu}(\Omega_{\Phi_{\chi}})\to\L^{2}_{\eps}(\Omega_{\Phi_{\chi}});$$
\item we consider the operator $\TAc :\R_{\gat}(\om) \to \R_{\gat}(\om)'_{\sharp}$ defined by
\begin{equation*}
\begin{split}
\scp{\TAc F}{G}
= \scp{\A_\chi F}{\A_\chi G}_{\L^{2}_{\mu_{\Pc}}(\om)}
&=  \bscp{ \mu_{\Pc}^{-1}\rot F}{ \rot G}_{\L^{2}(\om)} \\
&=\bscp{(\det J_{\Phi_\chi})^{-1}\wt{\mu}^{-1}J_{\Phi_\chi}\rot F}{J_{\Phi_\chi}\rot G}_{\L^{2}(\om)}
\end{split}
\end{equation*}
for all $F,G \in \R_{\gat}(\om)$; recall that  $\wt \mu =\mu \circ \Phi_{\chi}$;
\item we also consider  the operator $\mathcal{I}_\chi: \L^2(\Omega) \to \R_{\gat}(\om)'_{\sharp}$ defined by
\begin{equation*}
\scp{\mathcal{I}_\chi  F}{G} = \scp{F}{G}_{\L^{2}_{\eps_{\Pc}}(\om)} 
= \scp{\eps_{\Pc}F}{G}_{\L^{2}(\om)}
= \bscp{(\det J_{\Phi_\chi})\wt{\eps}J_{\Phi_\chi}^{-\top}F}{J_{\Phi_\chi}^{-\top}G}_{\L^{2}(\om)}
\end{equation*}
for all $F\in\L^2(\Omega)$ and $G\in \R_{\gat}(\om)$; recall that  $\wt \eps =\eps \circ \Phi_{\chi}$.
\end{itemize}

For the convenience of the reader, we recall that the transformation
\begin{equation*}
\tau^1_{\Phi_{\chi}}: \L^2_\eps(\Omega_{\Phi_{\chi}}) \to \L^2_{\eps_{\Phi_{\chi}}}(\Omega), \quad 
\tau^1_{\Phi_{\chi}}F = J_{\Phi_{\chi}}^\top \wt F
\end{equation*}
is an isometry, and if  $F \in D(\A)= \R_{\gatpchi}(\Omega_{\Phi_{\chi}})$,
 then $\tau^1_{\Phi_{\chi}}F \in  D(\A_{\chi}) = \R_{\gat}(\om)$ and 
\begin{equation*}
\rot  \tau^1_\Phi F = \tau^2_\Phi \rot F= (\adj J_\Phi)\widetilde{\rot T}.
\end{equation*} 

As in the previous section,  for a quantity $f_{\chi}$ depending on $\chi\in \mathcal X$,  we shall  use the symbol ${\partial _{\dot \chi}f_{\chi}}_{\vert_{\chi=\bar\chi}}$ (or  simply ${\partial _{\dot \chi}f_{\chi}}$)   to denote  the directional derivative of $f_{\chi}$ at  the fixed point $\bar \chi\in \mathcal X$, in the direction $\dot\chi\in X$.   Moreover,  for the following lemma we use the notation in \eqref{dzeta} and \eqref{dzetamu}. 

\begin{lem} \label{lemmaderchi} Assume that the maps $\eps, \mu : \mathbb{R}^3 \to \mathbb{R}^{3 \times 3}$, as well as
 the map $\chi \mapsto \Phi_{\chi}$,
are  of  class $\C^r$ with $r\in \nat \cup\{\infty\} \cup\{\omega\}$. 
Then the maps 
$\chi \mapsto \scp{   \mathcal{I}_\chi F  }{   G   }$ and $\chi \mapsto \scp{\mathcal{T}_\chi F }{G} $ are of class $\C^r$ and 
the following statements hold:
\begin{itemize}
\item[(i)] For all $F,G \in \L^2(\Omega)$ 
\begin{equation} \label{derIcal}
{\partial_{\dot \chi} \scp{   \mathcal{I}_\chi F  }{   G   } }_{\rvert_{\chi=\bar \chi}}= \scp{ \left(
\p_{\Psi}\eps+ (\div\Psi) \eps -2\sym(J_{\Psi}\eps) \right) \tau_{\Phi_{\bar\chi}^{-1}}^1 F  }{   \tau_{\Phi_{\bar\chi}^{-1}}^1 G   }_{\L^{2}(\Omega_{\Phi_{\bar\chi}})}.
\end{equation}
\item[(ii)] For all $F,G \in \R_{\gat}(\Omega)$
\begin{equation} \label{derTau}
{\partial_{\dot\chi}\scp{\mathcal{T}_\chi F }{G}}_{\rvert_{\chi=\bar \chi}} = \scp{ \left( \p_{\Psi}\mu^{-1}-(\mu^{-1}\div\Psi)+2\sym(\mu^{-1}J_{\Psi}) \right) \rot_{\ga_{\!t,\Phi_{\bar\chi}} } \tau_{\Phi_{\bar\chi}^{-1}}^1 F}{\rot_{\ga_{\!t,\Phi_{\bar\chi}} } \tau_{\Phi_{\bar\chi}^{-1}}^1 G}_{\L^{2}(\Omega_{\Phi_{\bar\chi}})}.
\end{equation}
\end{itemize}
\end{lem}
\begin{proof}
The regularity is easily deduced by Lemma \ref{lemmafrechet}. By the appendix of the first part we have
$$
{\p_{\dot \chi} \eps_{\Phi_{\chi}}}_{\rvert_{\chi=\bar \chi}}
=(\det J_{\Phi_{\bar\chi}})J_{\Phi_{\bar\chi}}^{-1}\Big(
\p_{\dot\chi }\wt{\eps}+(\wt{\div\Psi})\wt{\eps}-2\sym(\wt{J_{\Psi}}\wt{\eps})
\Big)J_{\Phi_{\bar\chi}}^{-\top}
$$
and
$$
{\p_{\dot\chi   }\mu_{\Phi_{\chi}}^{-1}}_{\rvert_{\chi=\bar \chi}}
=(\det J_{\Phi_{\bar\chi}})^{-1}J_{\Phi_{\bar\chi}}^{\top}\Big(
\p_{\dot\chi}\wt{\mu^{-1}}-(\wt{\div\Psi})\wt{\mu^{-1}}+2\sym(\wt{\mu^{-1}}\wt{J_{\Psi}})
\Big)J_{\Phi_{\bar\chi}}.
$$
Since
\begin{equation*}
\partial_{\dot\chi} \scp{   \mathcal{I}_\chi F  }{G} = \scp{\partial_{\dot\chi} \eps_{\Pc}F}{G}_{\L^{2}(\om)} \qquad \text{and} \qquad \partial_{\dot\chi}\scp{\mathcal{T}_\chi F}{G} =  \bscp{ \partial_{\dot\chi }\mu_{\Pc}^{-1}\rot F}{ \rot G}_{\L^{2}(\om)},
\end{equation*}
by changing variables and recalling that
\begin{equation*}
\tau^1_{\Phi^{-1}} F = \left( J_{\Phi}^{-\top} F \right) \circ \Phi^{-1}, \qquad \rot \tau^1_{\Phi^{-1}} F = \tau^2_{\Phi^{-1}} \rot F = \left((\det J_{\Phi})^{-1} J_{\Phi} \rot F \right) \circ \Phi^{-1}.
\end{equation*}
we conclude.
\end{proof}

By combining Theorems~\ref{symtheo} and Lemma~\ref{lemmaderchi}  we immediately deduce the following theorem. Note that no additional boundary regularity assumptions are used here. 
\begin{theo}\label{symtheomax}
Let $r\in \nat \cup\{\infty\}\cup\{\omega\}$ and  let $\eps$, $\mu$, and $\nu$ be of class $\C^r$. 
Let $\Omega$ be a bounded Lipschitz domain in $\mathbb{R}^3$ and let  $\{\Pc\}_{\chi \in \mathcal{X}}$
 be a family of  bi-Lipschitz transformations   $\Phi_\chi \in \mathcal{L}(\Omega)$  such that the map $\chi \mapsto \Phi_{\chi}$
is of class $\C^r$.  Let $\bar \chi\in \mathcal{X}$ be fixed.  Let $\bar \lambda$ be an eigenvalue of   the  Maxwell problem on $\Omega_{\Phi_{\bar\chi}}$. Assume $\bar \lambda$ has multiplicity $m$ and let
 $\bar\lambda=\lambda_{\bar\chi,k}=\dots = \lambda_{\bar\chi, k+m-1}$ for some $k \in \nat$. Let $F=\{k,k+1, \dots, k+m-1\}$ and let $\Lambda_{F,s}$ be the elementary symmetric functions of the Maxwell eigenvalues   defined as in \eqref{symfunct}. Then the functions $\Lambda_{F,s}$ are of class $\C^r$ in a neighborhood of $\bar\chi$ and  their directional derivatives   at the point $\bar \chi$ in the direction $\dot\chi$  are given by the  formula
\begin{equation*} 
\begin{split} 
& \bar \lambda ^{s-1}\binom{m-1}{s-1}\sum_{h=1}^m
\Big( \scp{ \left( \p_{\Psi}\mu^{-1}-(\mu^{-1}\div\Psi)+2\sym(\mu^{-1}J_{\Psi}) \right) \rot_{\ga_{\!t,\Phi_{\bar\chi}}} E_{\bar\chi,h}}{\rot_{\ga_{\!t,\Phi_{\bar\chi}}} E_{\bar\chi,h}}_{\L^{2}(\Omega_{\Phi_{\bar \chi}})} \\
&\qquad\quad\qquad\qquad\qquad\quad\qquad\qquad  - \bar \lambda \scp{ \left(
\p_{\Psi}\eps+ (\div\Psi) \eps -2\sym(J_{\Psi}\eps) \right)  E_{\bar\chi,h}  }{   E_{\bar\chi,h}   }_{\L^{2}(\Omega_{\Phi_{\bar\chi}})}\Big) \, ,
\end{split}
\end{equation*}
where  $\{E_{\bar\chi,h}\}_{h=1,\dots,m}$ is an orthonormal basis   in $\L^{2}_{\eps}(\Omega_{\Phi_{\bar\chi}})$ of the eigenspace  associated with $\bar\lambda$, and $\Psi$ is as in \eqref{defpsi}.
\end{theo}

By Theorem~\ref{Rellichthm}   we also deduce the following theorem where in particular we recover the formulas in our Theorem~4.1 in Part I of this series of papers  when $\bar \lambda$ is simple. 

\begin{theo}  \label{theo:shapederivatives}
Let the same assumptions of Theorem~\ref{symtheomax} hold. Assume that $r=\{1,\omega\}$ and $X={\mathbb{R}}$.
Then for $\chi$ sufficiently close to $\bar \chi$, the eigenvalues  $\lambda_{\chi, k}, \dots ,$ $  \lambda_{\chi, k+m-1}$ can be represented by $m$      functions $\chi \mapsto \zeta_{\chi, k}, \dots , \zeta_{\chi, k+m-1}$ of class $\C^r$.  
Moreover the derivatives of those functions at $\bar \chi$ are given by the eigenvalues of the following matrix
\begin{equation} \label{formula:shapederivatives}
\begin{split}
&\Big( \scp{ \left( \p_{\Psi}\mu^{-1}-(\mu^{-1}\div\Psi)+2\sym(\mu^{-1}J_{\Psi}) \right) \rot_{\ga_{\!t,\Phi_{\bar\chi}}} E_{\bar\chi,h}}{\rot_{\ga_{\!t,\Phi_{\bar\chi}}} E_{\bar\chi,l}}_{\L^{2}(\Omega_{\Phi_{\bar \chi}})}  \\
& \qquad - \bar \lambda \scp{ \left(
\p_{\Psi}\eps+ (\div\Psi) \eps -2\sym(J_{\Psi}\eps) \right)  E_{\bar\chi,h}  }{   E_{\bar\chi,l}   }_{\L^{2}(\Omega_{\Phi_{\bar\chi}})} \Big)_{h,l=1,\dots,m},
\end{split}
\end{equation}
where  $\{E_{\bar\chi,h}\}_{h=1,\dots,m}$ is an orthonormal basis   in $\L^{2}_{\eps}(\Omega_{\Phi_{\bar\chi}})$ of the eigenspace  associated with $\bar\lambda$, and $\Psi$ is as in \eqref{defpsi}. 
\end{theo}

\begin{proof}
 Recall from Part I of this series of papers that  $F \in \R_{\gat}(\Omega)$ is an eigenvector of the operator  $\A^*_{\bar\chi}\A_{\bar\chi}$ associated with an eigenvalue  $\bar\lambda$  if and only if  $E=\tau_{\Phi^{-1}_{\bar \chi}}^1 F$ belongs to $\R_{\Gamma_{\!t,\Phi_{\bar \chi}}}(\Omega_{\Phi{\bar \chi}})$ and is an eigenvector of the operator $\A^*\A$ associated with the same eigenvalue $\bar\lambda$. Moreover, 
 $F_{\bar\chi , h}$, $h=1,\dots,m$ is an orthonormal  basis  in $\L^2_{\eps_{\Phi_{\bar \chi}}}(\Omega)$ of the eigenspace  of $\A_{\bar\chi}^*\A_{\bar \chi}$ associated with $\bar \lambda$ if and only if  $E_{\bar\chi,h}=\tau_{\Phi^{-1}_{\bar \chi}   }^1 F_{\bar\chi, h}$, $h=1,\dots,m$ is an orthonormal basis   in $\L^2_{\eps}(\Omega_{\Phi_{\bar\chi}})$ of the eigenspace of   $\A^*\A$ associated with $\bar\lambda$. 
By Lemma \ref{lemmaderchi} we have that
\begin{eqnarray}\label{derivativeinteg}
\lefteqn{
\partial_{\dot\chi}  \scp{    \TAc \ F_{  \bar\chi, h}  }{   F_{   \bar\chi, l}   }-  \bar\lambda  \partial_{\dot\chi} \scp{   \mathcal{I}_\chi F_{   \bar\chi, h}  }{   F_{   \bar\chi, l} } 
}\nonumber \\ 
& & = \scp{ \left( \p_{\Psi}\mu^{-1}-(\mu^{-1}\div\Psi)+2\sym(\mu^{-1}J_{\Psi}) \right) \rot_{\ga_{\!t,\Phi_{\bar\chi}}} E_{\bar\chi,h}}{\rot_{\ga_{\!t,\Phi_{\bar\chi}}} E_{\bar\chi,l}}_{\L^{2}(\Omega_{\Phi_{\bar \chi}})} \\
& & \quad - \bar \lambda \scp{ \left(
\p_{\Psi}\eps+ (\div\Psi) \eps -2\sym(J_{\Psi}\eps) \right)  E_{\bar\chi,h}  }{   E_{\bar\chi,l}   }_{\L^{2}(\Omega_{\Phi_{\bar\chi}})} \nonumber.
\end{eqnarray}
The proof follows by applying Theorem~\ref{Rellichthm}. 
\end{proof}

In order to write the formulas in the two previous statements by means of surface integrals we need additional assumptions as follows. 

\begin{cor}[Hirakawa's formula] \label{coroll:surfaceinthe} Assume that $\bar \lambda$ and  $\{E_{\bar\chi,h}\}_{h=1,\dots,m}$ are as in  Theorem~\ref{symtheomax} and assume that  the Gaffney inequalities \eqref{gaffneytoomuchelec}, \eqref{gaffneytoomuchmagn} hold in $\Omega_{\Phi_{\bar \chi}}$. 
Then the following statements hold. 
\begin{itemize}
\item[(i)]  The directional derivatives   at the point $\bar \chi$ in the direction $\dot\chi$ of the functions $\Lambda_{F,s}$ from Theorem~\ref{symtheomax} can be written as 
\begin{equation}\label{hirakawasym}
\begin{split}
& \bar \lambda ^{s-1}\binom{m-1}{s-1}\sum_{h=1}^m
\Bigg(
\int_{\Gamma_{n,\Phi_{\bar\chi}}} \left( (\mu^{-1} \rot E_{\bar\chi,h} \cdot \rot \ol{E_{\bar\chi,h}} )  -  \lambda (\eps E_{\bar\chi,h} \cdot \ol{E_{\bar\chi,h}} ) \right) (n \cdot \Psi)  d\sigma \\
& \qquad \qquad - \int_{\Gamma_{t,\Phi_{\bar\chi}}} \left( (\mu^{-1} \rot E_{\bar\chi,h} \cdot \rot \ol{E_{\bar\chi,h}} )  -  \lambda (\eps E_{\bar\chi,h} \cdot \ol{E_{\bar\chi,h}} ) \right) (n \cdot \Psi)  d\sigma
\Bigg) .
\end{split}
\end{equation}
\item[(ii)]
The derivatives   at a point  $\bar \chi\in \mathbb{R}$  of the  functions $\zeta_{\chi, k}, \dots , \zeta_{\chi, k+m-1}$  from  Theorem~\ref{theo:shapederivatives} are given by the eigenvalues of the following matrix:
\begin{equation}\label{hirakawa}
\begin{split}
\Bigg(
&\int_{\Gamma_{n,\Phi_{\bar\chi}}} \left( (\mu^{-1} \rot E_{\bar\chi,h} \cdot \rot \ol{E_{\bar\chi,l}} )  -  \lambda (\eps E_{\bar\chi,h} \cdot \ol{E_{\bar\chi,l}} ) \right) (n \cdot \Psi)  d\sigma \\
& \qquad \qquad - \int_{\Gamma_{t,\Phi_{\bar\chi}}} \left( (\mu^{-1} \rot E_{\bar\chi,h} \cdot \rot \ol{E_{\bar\chi,l}} )  -  \lambda (\eps E_{\bar\chi,h} \cdot \ol{E_{\bar\chi,l}} ) \right) (n \cdot \Psi)  d\sigma
\Bigg)_{h,l=1,\dots,m},
\end{split}
\end{equation}
.
\end{itemize}
\end{cor}

The proof of Corollary  \ref{coroll:surfaceinthe} is an immediate application of Theorems \ref{symtheomax},\ref{theo:shapederivatives} and the following lemma.

\begin{lem} \label{lem:maxsurfaceint}
Suppose that  the  Gaffney inequalities \eqref{gaffneytoomuchelec}, \eqref{gaffneytoomuchmagn} hold   in $\Omega_{\Phi_{\bar\chi}}$. 
Suppose that $F,G \in \L^2_{\eps}(\Omega_{\Phi_{\bar\chi}})$ are Maxwell eigenfunctions associated with an eigenvalue  $\bar \lambda \ne 0$. 
Then 
\begin{equation*}
\begin{split}
&\scp{ \left( \p_{\Psi}\mu^{-1}-(\mu^{-1}\div\Psi)+2\sym(\mu^{-1}J_{\Psi}) \right) \rot_{\ga_{\!t,\Phi_{\bar\chi}}} F }{\rot_{\ga_{\!t,\Phi_{\bar\chi}}} G }_{\L^{2}(\Omega_{\Phi_{\bar \chi}})} \\
& \qquad \qquad - \bar \lambda \scp{ \left(
\p_{\Psi}\eps+ (\div\Psi) \eps -2\sym(J_{\Psi}\eps) \right)  F  }{  G   }_{\L^{2}(\Omega_{\Phi_{\bar\chi}})} \\
&= \int_{\Gamma_{n,\Phi_{\bar\chi}}} \left( (\mu^{-1} \rot F \cdot \rot \ol G)  -  \lambda (\eps F \cdot \ol G) \right) (n \cdot \Psi)  d\sigma \\
& \qquad \qquad - \int_{\Gamma_{t,\Phi_{\bar\chi}}} \left( (\mu^{-1} \rot F \cdot \rot \ol G)  -  \lambda (\eps F \cdot \ol G) \right) (n \cdot \Psi)  d\sigma,
\end{split}
\end{equation*}
\end{lem}

\begin{proof}
To ease the notation, throughout this proof we write  $\Phi, \lambda$ instead of $\Phi_{\bar\chi}, \bar \lambda$, and we use the Einstein notation omitting summation symbols.  

Observe also that the quantities $\Psi, \eps,\mu$ have  real-valued entries, thus they coincide with their complex conjugates. The complex conjugate of $G$ is denoted by $\overline{G}$.

Recall that since $F,G$ are eigenvectors  of $A^* A$ associated with a non-zero eigenvalue, they belong to $D(A) \cap R(A^*) \subset D(\rot_{\Gamma_{t,\Phi}}) \cap N(\div_{\ganp}\eps)$ by the Hilbert complex property (see Remarks 2.21 and 3.10 of Part I). Therefore $\div(\eps F) = \div(\eps G)=0$ in $\Omega_\Phi$ and $\eps F \cdot n = \eps \, G \cdot n=0$ on $\Gamma_{n,\Phi}$.

Moreover, since by Lemma 2.15 of Part I, the vector fields $\mu^{-1}\rot F,\mu^{-1}\rot G$ are eigenvectors of $A A^*$ (associated with the same non-zero eigenvalue), they belong to $D(A^*) \cap R(A) \subset  D(\rot_{\Gamma_{n,\Phi}}) \cap N(\div_{\Gamma_{t,\Phi}}\mu)$ (see the operator $A_2$ in the Remark 2.21 of Part I), hence in particular $n \cdot \rot F= n \cdot \rot G=0$  on $\Gamma_{t,\Phi}$. As a consequence, by the Gaffney inequality $\rot F, \rot \ol G \in   \mu\,\R_{\ganp}(\Omega_\Phi) \cap \D_{\gatp} (\Omega_\Phi)\subset \H^1(\Omega_\Phi)$.

We set
\begin{equation*}
Q:=\int_{\Omega_\Phi} \left( \partial_\Psi \eps - 2 \sym(J_{\Psi}\eps) \right) F \cdot \overline{G}\, dy.
\end{equation*}
Observe that since by hypothesis the Gaffney inequality holds in $\Omega_\Phi$, then $F,G \in \H^1(\Omega_\Phi)$. 
Integrating by parts  yield
\begin{equation*}
\begin{split}
&-\int_{\Omega_\Phi} \left( J_\Psi \, \eps +(J_\Psi \, \eps)^\top \right) F \cdot \ol G=-\int_{\Omega_\Phi} F_i (\partial_k \Psi_i) \eps_{kj} \ol G_j  - \int_{\Omega_\Phi} F_i (\partial_k \Psi_j) \eps_{ki} \ol G_j   \\
&\quad = \int_{\Omega_\Phi} (\partial_k F_i) \Psi_i \eps_{kj} \ol G_j + \int_{\Omega_\Phi} F_i \Psi_j \eps_{ki} (\partial_k \ol G_j)  + \int_{\Omega_\Phi} F \cdot \Psi \divergence(\eps \ol G) + \int_{\Omega_\Phi} \ol G \cdot \Psi  \divergence(\eps F) \\
&\qquad - \int_{\Gamma_\Phi} (F \cdot \Psi) (\eps \ol G \cdot n) d\sigma  - \int_{\Gamma_\Phi} (\ol G \cdot \Psi) (\eps F \cdot n) d\sigma \\
&\quad = \int_{\Omega_\Phi} (\partial_k F_i) \Psi_i \eps_{kj} \ol G_j + \int_{\Omega_\Phi} F_i \Psi_j \eps_{ki} (\partial_k \ol G_j) - \int_{\Gamma_\Phi} (F \cdot \Psi) (\eps \ol G \cdot n) d\sigma  - \int_{\Gamma_\Phi} (\ol G \cdot \Psi) (\eps F \cdot n) d\sigma .
\end{split}
\end{equation*} 
Furthermore we have that
\begin{equation*}
\sum_{i,j,k=1}^3 \left((\partial_k F_i)  - (\partial_i F_k) \right) \eps_{kj} \ol G_j  \Psi_i= \rot F \cdot (\eps \ol G \times \Psi)
\end{equation*}
Moreover observe that
\begin{equation*}
\begin{split}
\nabla (\eps F \cdot \ol G) \cdot \Psi &= (\partial_i F_k) \eps_{kj} \ol G_j \Psi_i + (\partial_i \ol G_j)\eps_{kj} F_k \Psi_i + (\partial_i \eps_{kj})F_k \ol G_j \Psi_i \\
&= (\partial_i F_k) \eps_{kj} \ol G_j \Psi_i+ (\partial_i \ol G_j)\eps_{kj} F_k \Psi_i+ F_k \ol G_j(\nabla \eps_{kj} \cdot \Psi)_{1\leq k,j \leq 3} \\
&= (\partial_i F_k) \eps_{kj} \ol G_j \Psi_i + (\partial_i \ol G_j)\eps_{kj} F_k  \Psi_i+ F \cdot (\partial_\Psi \eps)\ol G. \\
& = (\partial_i F_k) \eps_{kj} \ol G_j \Psi_i + (\partial_i \ol G_j)\eps_{kj} F_k  \Psi_i+ (\partial_\Psi \eps) F \cdot \ol G.
\end{split}
\end{equation*} 
Therefore
\begin{equation} \label{eq:Q1}
\begin{split}
Q&= \int_{\Omega_\Phi} \rot F \cdot (\eps \ol G \times \Psi) + \rot \ol G \cdot(\eps F \times \Psi)  + \nabla (\eps F \cdot \ol G) \cdot \Psi \\
 & \qquad - \int_{\Gamma_\Phi} \left( (F \cdot \Psi) (\eps \ol G \cdot n) + (\ol G \cdot \Psi) (\eps F \cdot n)\right) d\sigma \\
& = \int_{\Omega_\Phi} \eps \ol G \cdot (\Psi  \times \rot F) + \eps F \cdot (\Psi  \times \rot \ol G)  + \nabla (\eps F \cdot \ol G) \cdot \Psi - 2\int_{\Gamma_{t,\Phi}}   (\eps F \cdot \ol G) (n \cdot \Psi)d\sigma,
\end{split}
\end{equation}
where in the last equality  we have used the fact that $\eps \, F \cdot n=\eps \, \ol G \cdot n=0$ on $\Gamma_{n,\Phi}$, while $F=(F \cdot n)n$ and $\ol G=(\ol G \cdot n)n$ on $\Gamma_{t,\Phi}$.
By recalling that  $F$ and $G$ are eigenvectors of $\A^*\A = \eps^{-1}\rot \mu^{-1} \rot_{\Gamma_\Phi}$ and integrating by parts we get 
\begin{equation} \label{eq:Q2}
\begin{split}
&\int_{\Omega_\Phi} \eps \ol G \cdot (\Psi  \times \rot F) = \lambda^{-1}\int_{\Omega_\Phi}  \rot \mu^{-1} \rot \ol G \cdot (\Psi \times \rot F) \\
&=\lambda^{-1} \int_{\Omega_\Phi}  \mu^{-1} \rot \ol G \cdot \rot(\Psi \times \rot F) + \lambda^{-1} \int_{\Gamma_\Phi} (n \times \mu^{-1}\rot \ol G) \cdot (\Psi \times \rot F) d\sigma\\
&=-\lambda^{-1} \int_{\Omega_\Phi}  \mu^{-1} \rot \ol G \cdot \rot F \divergence\Psi + \lambda^{-1} \int_{\Omega_\Phi}  \mu^{-1} \rot \ol G \cdot J_\Psi \rot F \\
&\quad - \lambda^{-1} \int_{\Omega_\Phi} \mu^{-1} \rot \ol G \cdot J_{\rot F} \Psi  + \lambda^{-1} \int_{\Gamma_{t,\Phi}}  (\mu^{-1} \rot F \cdot \rot \ol G) (n \cdot \Psi) d\sigma  .
\end{split}
\end{equation} 
Here we  have used the fact that $\rot(\Psi \times \rot F) = - \rot F \, \divergence\Psi + J_\Psi \rot F - J_{\rot F} \Psi$ and that
\begin{equation*}
n \times \mu^{-1} \rot F = n \times \mu^{-1} \rot \ol G =0 
\end{equation*}
on $\Gamma_{n,\Phi}$, while on $\Gamma_{t,\Phi}$ we have that
\begin{equation*}
\begin{split}
(n \times \mu^{-1}\rot \ol G) \cdot (\Psi \times \rot F)&=  \Psi \cdot \left( \rot F \times ( n \times \mu^{-1} \rot \ol G) \right)\\
&= \Psi \cdot \left( (\rot F \cdot \mu^{-1} \rot \ol G) n - (\rot F \cdot n) \mu^{-1} \rot
\ol G \right) \\
&= (\rot F \cdot \mu^{-1} \rot \ol G) (n \cdot \Psi)\\
&= (\mu^{-1} \rot F \cdot \rot \ol G) (n \cdot \Psi)
\end{split}
\end{equation*}
by Lagrange's formula for the vector triple product and the fact that $n \cdot \rot F=0$  on $\Gamma_{t,\Phi}$.
An analogous equality holds for 
\begin{equation} \label{eq:Q3}
\int_{\Omega_\Phi} \eps F \cdot (\Psi  \times \rot \ol G).
\end{equation} 
Observe that the Jacobians matrices $J_{\rot F}, J_{\rot \ol G}$ are well defined since  $\rot F, \rot \ol G \in  \H^1(\Omega_\Phi)$.   
By standard calculus and an approximation argument we also have that
\begin{equation} \label{eq:nablarotid}
\begin{split}
\nabla(\mu^{-1} \rot F \cdot \rot \ol G) \cdot \Psi = J_{\rot \ol G} \Psi \cdot \mu^{-1} \rot F +J_{\rot F} \Psi \cdot \mu^{-1} \rot \ol G +  (\partial_\Psi \mu^{-1})\rot F \cdot \rot \ol G.
\end{split}
\end{equation}
Hence by \eqref{eq:Q1}--\eqref{eq:nablarotid}
\begin{equation} \label{eqlambdaA}
\begin{split}
\lambda  Q &= - 2\int_{\Omega_\Phi} \mu^{-1}  \rot F  \cdot \rot \ol G \divergence \Psi + \int_{\Omega_\Phi} \left( \mu^{-1} J_\Psi + (\mu^{-1} J_\Psi)^\top \right) \rot F \cdot \rot \ol G  \\
&\quad -\int_{\Omega_\Phi}  \nabla (\mu^{-1} \rot F \cdot \rot \ol G) \cdot \Psi + \int_{\Omega_\Phi} (\partial_\Psi \mu^{-1})(\rot F) \cdot \rot \ol G + \lambda \int_{\Omega_\Phi} \nabla(\eps F \cdot \ol G) \cdot \Psi \\
&\quad +2 \int_{\Gamma_{t,\Phi}}  (\mu^{-1} \rot F \cdot \rot \ol G) (n \cdot \Psi)  d\sigma - 2 \lambda \int_{\Gamma_{t,\Phi}}   (\eps F \cdot \ol G) (n \cdot \Psi)  d\sigma.
\end{split}
\end{equation}
Applying the divergence theorem we obtain
\begin{equation} \label{eq:finalQ}
\begin{split}
\lambda Q &=- \int_{\Omega_\Phi} \mu^{-1} \rot F \cdot \rot \ol G \divergence \Psi + \int_{\Omega_\Phi} \left( \mu^{-1} J_\Psi + (\mu^{-1} J_\Psi)^\top \right) \rot F \cdot \rot \ol G \\
&\quad + \int_{\Omega_\Phi} (\partial_\Psi \mu^{-1}) (\rot F) \cdot \rot \ol G  -\lambda \int_{\Omega_\Phi} (\eps F \cdot \ol G) \divergence\Psi \\
&\quad - \int_{\Gamma_\Phi} (\mu^{-1} \rot F \cdot \rot \ol G)  (n \cdot \Psi)  d\sigma +\lambda \int_{\Gamma_\Phi}  (\eps F \cdot \ol G) (n \cdot \Psi)  d\sigma \\
&\quad +2 \int_{\Gamma_{t,\Phi}} (\mu^{-1} \rot F \cdot \rot \ol G) (n \cdot \Psi)  d\sigma - 2 \lambda \int_{\Gamma_{t,\Phi}}  (\eps F \cdot \ol G) (n \cdot \Psi)  d\sigma\\
&=- \int_{\Omega_\Phi} \mu^{-1} \rot F \cdot \rot \ol G \divergence \Psi + \int_{\Omega_\Phi} \left( \mu^{-1} J_\Psi + (\mu^{-1} J_\Psi)^\top \right) \rot F \cdot \rot \ol G \\
&\quad + \int_{\Omega_\Phi} (\partial_\Psi \mu^{-1}) (\rot F) \cdot \rot \ol G  -\lambda \int_{\Omega_\Phi} (\eps F \cdot \ol G) \divergence\Psi \\
&\quad - \int_{\Gamma_{n,\Phi}}  (\mu^{-1} \rot F \cdot \rot \ol G)  (n \cdot \Psi)  d\sigma +\lambda \int_{\Gamma_{n,\Phi}}   (\eps F \cdot \ol G) (n \cdot \Psi)  d\sigma \\
&\quad +\int_{\Gamma_{t,\Phi}}  (\mu^{-1} \rot F \cdot \rot \ol G)  (n \cdot \Psi) d\sigma -  \lambda \int_{\Gamma_{t,\Phi}}  (\eps F \cdot \ol G) (n \cdot \Psi)  d\sigma.
\end{split}
\end{equation}
Finally, from \eqref{eq:finalQ} we deduce that
\begin{equation*}
\begin{split}
&\scp{ \left( \p_{\Psi}\mu^{-1}-(\mu^{-1}\div\Psi)+2\sym(\mu^{-1}J_{\Psi}) \right) \rot F}{\rot  G}_{\L^{2}(\Omega_{\Phi})} \\
& \qquad \qquad - \lambda \scp{ \left(
\p_{\Psi}\eps+ (\div\Psi) \eps -2\sym(J_{\Psi}\eps) \right)  F  }{   G  }_{\L^{2}(\Omega_{\Phi})} \\
&=  \int_{\Omega_\Phi}  \left( \partial_\Psi \mu^{-1}  - \mu^{-1} \div \Psi +\mu^{-1} J_\Psi + (\mu^{-1} J_\Psi)^\top\right) \rot F \cdot \rot \ol G -\lambda Q - \lambda \int_{\Omega_\Phi}  \divergence \Psi \, \eps F \cdot \ol G  \\
&= \int_{\Gamma_{n,\Phi}} \left( (\mu^{-1} \rot F \cdot \rot \ol G)  -  \lambda (\eps F \cdot \ol G) \right) (n \cdot \Psi)  d\sigma \\
& \qquad \qquad - \int_{\Gamma_{t,\Phi}} \left( (\mu^{-1} \rot F \cdot \rot \ol G)  -  \lambda (\eps F \cdot \ol G) \right) (n \cdot \Psi)  d\sigma,
\end{split}
\end{equation*}
and the proof is complete.
\end{proof}

\subsection{Helmholtz Eigenvalues} \label{helmshapeder}

In this subsection we focus on the Helmholtz eigenproblem \eqref{helmprobl}.
We follow the notation and the results of  Section~\ref{sec:shape2sound} by specialising the function spaces and the operators as follows:
\begin{itemize}
\item  we fix the ambient Hilbert spaces  by setting $\H_0 =\L^2(\Omega)$, $\H_1= \L^2(\Omega)$ (with $\H_1$ being a space of vector valued functions, i.e.  $\H_1= \H_0^3$) and we consider also the 
Hilbert spaces  $\H_{0,\chi} =\L^{2}_{\nu_{\Phi_\chi}}(\om)$ and $\H_{1,\chi} = \L^{2}_{\eps_{{\Phi}_{\chi}}}(\om)$  depending on $\chi$.   Note that changing the value of $\chi$ gives equivalent scalar products;
\item  we consider the operator 
$$\A_{\chi}:=\na_{\gat}:\H^{1}_{\gat}(\om)\subset\L^{2}_{\nu_{\Phi_{\chi}}}(\om)\to\L^{2}_{\eps_{\Phi_{\chi}}}(\om)$$
as the pull-back of the operator
$$\A:=\na_{\gatpchi}:\H^{1}_{\gatpchi}(\Omega_{\Phi_{\chi}})\subset\L^{2}_{\nu}(\Omega_{\Phi_{\chi}})\to\L^{2}_{\eps}(\Omega_{\Phi_{\chi}});$$
recall  their adjoints
$$\A_{\chi}^{*}=-\nu_{\Phi_{\chi}}^{-1}\div_{\gan}\eps_{\Phi_{\chi}}:\eps_{\Phi_{\chi}}^{-1}\D_{\gan}(\om)\subset\L^{2}_{\eps_{\Phi_{\chi}}}(\om)\to\L^{2}_{\nu_{\Phi_{\chi}}}(\om),$$
$$\A^{*}=-\nu^{-1}\div_{\ganpchi}\eps:\eps^{-1}\D_{\ganpchi}(\Omega_{\Phi_{\chi}})\subset\L^{2}_{\eps}(\Omega_{\Phi_{\chi}})\to\L^{2}_{\nu}(\Omega_{\Phi_{\chi}});$$ 
\item we consider the operator $\TAc :\H^{1}_{\gat}(\om) \to \H^{1}_{\gat}(\om)'_{\sharp}$ defined by
\begin{equation*}
\begin{split}
\scp{\TAc f }{g}&= \scp{\A_\chi f}{\A_\chi g}_{\L^{2}_{\mu_{\Pc}}(\om)} 
= \bscp{\nabla f}{ \nabla g}_{\L^{2}_{\eps_{\Phi_\chi}}(\om)} 
=  \bscp{ \eps_{\Phi_\chi} \nabla f}{ \nabla g}_{\L^{2}(\om)} \\
&= \bscp{(\det J_{\Phi_\chi})\wt{\eps}J_{\Phi_\chi}^{-\top}\nabla f}{J_{\Phi_\chi}^{-\top}\nabla g}_{\L^{2}(\om)}
\end{split}
\end{equation*}
for all $f,g \in \H^{1}_{\gat}(\om)$; recall  that  $\wt \eps =\eps \circ \Phi_{\chi}$;
\item we also consider  the operator $\mathcal{I}_\chi: \L^2(\Omega) \to \H^{1}_{\gat}(\om)'_{\sharp}$ defined by
\begin{equation*}
\scp{\mathcal{I}_\chi  f }{g} 
= \scp{ f}{ g}_{\L^{2}_{\nu_{\Pc}}(\om)} 
= \scp{\nu_{\Pc}f}{g}_{\L^{2}(\om)} 
=\bscp{(\det J_{\Phi_\chi})\wt{\nu}f}{g}_{\L^{2}(\om)}
\end{equation*}
for all $f \in \L^2(\Omega)$ and $g \in \H^{1}_{\gat}(\om)$; recall that  $\wt \nu =\nu \circ \Phi_{\chi}$.
\end{itemize}

We have the equivalent to Lemma \ref{lemmaderchi}. Recall that  $\Psi$ is defined in \eqref{defpsi}.

\begin{lem} \label{lemmaderchihelm} 
Assume that the map $\nu:\mathbb{R}^3 \to \mathbb{R}$ as well as
 the map $\chi \mapsto \Phi_{\chi}$
are  of  class $\C^r$ with $r\in \nat \cup\{\infty\} \cup\{\omega\}$.
Then $\chi \to \scp{   \mathcal{I}_\chi F  }{   G   }$ and $\chi \to \scp{\mathcal{T}_\chi F }{G} $ are of class $\C^r$ and the following statements hold:
\begin{itemize}
\item[(i)] For all $f,g \in \L^2(\Omega)$ 
\begin{equation}
{\partial_{\dot\chi} \scp{   \mathcal{I}_\chi f  }{  g  }}_{\vert_{\chi=\bar\chi}}
= \scp{ \left( \p_{\Psi}\nu+(\div\Psi)\nu \right) \tau_{\Phi_{\bar\chi}^{-1}}^0 f  }{   \tau_{\Phi_{\bar\chi}^{-1}}^0 g   }_{\L^{2}(\Omega_{\Phi_{\bar\chi}})}.
\end{equation}
\item[(ii)] For all $f,g \in \H^{1}_{\gat}(\om)$
\begin{equation} 
{\partial_{\dot\chi}\scp{\mathcal{T}_\chi f }{g}}_{\vert_{\chi=\bar\chi}}
= \scp{ \left(\p_{\Psi  }\eps+(\div\Psi)\eps-2\sym(J_{\Psi}\eps) \right) \na_{\ga_{\!t,\Phi_{\bar\chi}}} \tau_{\Phi_{\bar\chi}^{-1}}^0 f}{\na_{\ga_{\!t,\Phi_{\bar\chi}} } \tau_{\Phi_{\bar\chi}^{-1}}^0 g}_{\L^{2}(\Omega_{\Phi_{\bar\chi}})}.
\end{equation}
\end{itemize}
\end{lem}
\begin{proof}
The regularity is easily deduced by Lemma \ref{lemmafrechet}. In order to prove the formulas in the statement, 
it suffices to use the appendix of the first part, perform a change of variables, and  recall that 
\begin{equation*}
\tau^0_{\Phi^{-1}} f =  f  \circ \Phi^{-1}, \qquad \nabla \tau^0_{\Phi^{-1}} f
 = \tau^1_{\Phi^{-1}} \nabla f 
=  \left( J_{\Phi}^{-\top} \nabla f \right) \circ \Phi^{-1}.
\end{equation*}
\end{proof}

The following theorem is an immediate consequence of Theorem~\ref{symtheo} and Lemma~\ref{lemmaderchihelm}.
\begin{theo}\label{symtheomaxhe}
Let $r\in \nat \cup\{\infty\}\cup\{\omega\}$ and  let $\eps$, $\mu$, and $\nu$ be of class $\C^r$. 
Let $\Omega$ be a bounded Lipschitz domain in $\mathbb{R}^3$ and let  $\{\Pc\}_{\chi \in \mathcal{X}}$
 be a family of  bi-Lipschitz transformations   $\Phi_\chi \in \mathcal{L}(\Omega)$  such that the map $\chi \mapsto \Phi_{\chi}$
is of class $\C^r$.  Let $\bar \chi\in \mathcal{X}$ be fixed.  Let $\bar \lambda$ be an eigenvalue of  the Helmholtz problem  on $\Omega_{\Phi_{\bar\chi}}$. Assume $\bar \lambda$ has multiplicity $m$ and let
 $\bar\lambda=\lambda_{\bar\chi,k}=\dots = \lambda_{\bar\chi, k+m-1}$ for some $k \in \nat$. Let $F=\{k,k+1, \dots, k+m-1\}$ and let $\Lambda_{F,s}$ be the elementary symmetric functions of the Helmholtz ,  defined as in \eqref{symfunct}.   Then the functions $\Lambda_{F,s}$ are of class $\C^r$ in a neighborhood of $\bar\chi$ and  their directional derivatives   at the point $\bar \chi$ in the direction $\dot\chi$  are given by the  formula
\begin{equation*}
\begin{split} 
& \bar \lambda ^{s-1}\binom{m-1}{s-1}\sum_{h=1}^m
\Big( \scp{ \left(\p_{\Psi  }\eps+(\div\Psi)\eps-2\sym(J_{\Psi}\eps) \right) \na_{\ga_{\!t,\Phi_{\bar\chi}}} u_{\bar\chi,h}}{\na_{\ga_{\!t,\Phi_{\bar\chi}}} u_{\bar\chi,h} }_{\L^{2}(\Omega_{\Phi_{\bar \chi}})}\\
&\qquad\quad\qquad\qquad\qquad\quad\qquad
\qquad - \bar \lambda \scp{ \left( \p_{\Psi}\nu+(\div\Psi)\nu \right) u_{\bar\chi,h} }{ u_{\bar\chi,h}  }_{\L^{2}(\Omega_{\Phi_{\bar \chi}})} \Big) \, ,
\end{split}
\end{equation*}
where  $\{u_{\bar\chi,h}\}_{h=1,\dots,m}$ is an orthonormal basis   in $\L^{2}_{\nu}(\Omega_{\Phi_{\bar\chi}})$ of the eigenspace  associated with $\bar\lambda$, and $\Psi$ is as in \eqref{defpsi}.
\end{theo}

By Theorem~\ref{Rellichthm}    we also deduce the following theorem where in particular we recover the formulas in our Theorem~4.1 in Part I of this series of papers when $\bar \lambda$ is simple.

\begin{theo}  \label{theo:shapederivativeshe}
Let the same assumptions of Theorem~\ref{symtheomaxhe} hold. Assume that $r=\{1,\omega\}$ and $X={\mathbb{R}}$.
Then for $\chi$ sufficiently close to $\bar \chi$, the eigenvalues  $\lambda_{\chi, k}, \dots ,$ $  \lambda_{\chi, k+m-1}$ can be represented by $m$      functions $\chi \mapsto \zeta_{\chi, k}, \dots , \zeta_{\chi, k+m-1}$ of class $\C^r$.  
Moreover the derivatives of those functions at $\bar \chi$ are given by the eigenvalues of the following matrix
\begin{equation*}
\begin{split}
&\Big( \scp{ \left(\p_{\Psi  }\eps+(\div\Psi)\eps-2\sym(J_{\Psi}\eps) \right) \na_{\ga_{\!t,\Phi_{\bar\chi}}} u_{\bar\chi,h}}{\na_{\ga_{\!t,\Phi_{\bar\chi}}} u_{\bar\chi,l} }_{\L^{2}(\Omega_{\Phi_{\bar \chi}})} \\
&\qquad - \bar \lambda \scp{ \left( \p_{\Psi}\nu+(\div\Psi)\nu \right) u_{\bar\chi,h} }{ u_{\bar\chi,l}  }_{\L^{2}(\Omega_{\Phi_{\bar \chi}})} \Big)_{h,l=1,\dots,m},
\end{split}
\end{equation*}
where  $\{u_{\bar\chi,h}\}_{h=1,\dots,m}$ is an orthonormal basis   in $ \L^{2}_{\nu}(\Omega_{\Phi_{\bar\chi}})$ of the eigenspace  associated with $\bar\lambda$. 
\end{theo}

\begin{proof}
Recall from Part I that  $v \in \H^{1}_{\gat}(\om)$ is an eigenfunction of the operator  $\A^*_{\bar\chi}\A_{\bar\chi}$ associated with an eigenvalue  $\bar\lambda$  if and only if  $\tau_{\Phi^{-1}_{\bar \chi}}^1 v$ belongs to $\H^{1}_{\ga_{\!t,\Phi_{\bar\chi}}}(\Omega_{\Phi{\bar \chi}})$ and is an eigenfunction of the operator $\A^*\A$ associated with the same eigenvalue $\bar\lambda$. Moreover, 
$v_{\bar\chi , h}$, $h=1,\dots,m$ is an orthonormal  basis  in $\L^{2}_{\nu_{\Phi_{\chi}}}(\om)$ of the eigenspace  of $\A_{\bar\chi}^*\A_{\bar \chi}$ associated with $\bar \lambda$ if and only if  $u_{\bar\chi,h}=    \tau_{\Phi^{-1}_{\bar \chi}   }^0 v_{\bar\chi, h}$, $h=1,\dots,m$ is an orthonormal basis   in $\L^{2}_{\nu}(\Omega_{\Phi_{\chi}})$ of the eigenspace of   $\A^*\A$ associated with $\bar\lambda$. 
By Lemma \ref{lemmaderchihelm} we have that
\begin{eqnarray}\label{derivativeinteghelm}
\lefteqn{
\partial_{\dot\chi}  \scp{    \TAc \ v_{  \bar\chi, h}  }{   v_{   \bar\chi, l}   }-  \bar\lambda  \partial_{\dot\chi} \scp{   \mathcal{I}_\chi v_{   \bar\chi, h}  }{   v_{   \bar\chi, l} } 
}\nonumber \\ 
& & = \scp{ \left(\p_{\Psi  }\eps+(\div\Psi)\eps-2\sym(J_{\Psi}\eps) \right) \na_{\ga_{\!t,\Phi_{\bar\chi}}} u_{\bar\chi,h}}{\na_{\ga_{\!t,\Phi_{\bar\chi}}} u_{\bar\chi,l} }_{\L^{2}(\Omega_{\Phi_{\bar \chi}})} \\
& & \quad - \bar \lambda \scp{ \left( \p_{\Psi}\nu+(\div\Psi)\nu \right) u_{\bar\chi,h} }{ u_{\bar\chi,l}  }_{\L^{2}(\Omega_{\Phi_{\bar \chi}})} \nonumber
\end{eqnarray}
which combined with  Theorem~\ref{Rellichthm}  completes the proof.
\end{proof}


\begin{cor} \label{coroll:surfaceint} Assume that $\bar \lambda$ and  $\{u_{\bar\chi,h}\}_{h=1,\dots,m}$ are as in  Theorem~\ref{symtheomaxhe} and assume in addition that  $u_{\bar\chi,h} \in \H^2(\Omega_{\Phi_{\bar \chi}})$ for all $h=1,\dots,m$.
Then the following statements hold. 
\begin{itemize}
\item[(i)]  The directional derivatives   at the point $\bar \chi$ in the direction $\dot\chi$ of the functions $\Lambda_{F,s}$ from Theorem~\ref{symtheomaxhe} can be written as 
\begin{equation}
\begin{split}
 & \bar \lambda ^{s-1}\binom{m-1}{s-1}\sum_{h=1}^m
\Bigg(
\int_{\Gamma_{n,\Phi_{\bar\chi}}} \left( \eps \nabla u_{\bar\chi,h} \cdot \nabla \overline{u_{\bar\chi,h}} - \bar \lambda \, \nu \, u_{\bar\chi,h} \, \overline{u_{\bar\chi,h}} \right) (\Psi \cdot n) d\sigma \\
&\qquad \qquad \qquad  \qquad \qquad-\int_{\Gamma_{t,\Phi_{\bar\chi}}} (\eps \nabla u_{\bar\chi,h} \cdot \nabla \overline{u_{\bar\chi,h}}) (\Psi \cdot n) d\sigma
\Bigg).
\end{split}
\end{equation}
\item[(ii)]
The derivatives   at a point  $\bar \chi\in \mathbb{R}$  of the  functions $\zeta_{\chi, k}, \dots , \zeta_{\chi, k+m-1}$  from  Theorem~\ref{theo:shapederivativeshe} are given by the eigenvalues of the following matrix:
\begin{equation}\label{hirakawahe}
\begin{split}
& \Bigg( 
\int_{\Gamma_{n,\Phi_{\bar\chi}}} \left( \eps \nabla u_{\bar\chi,h} \cdot \nabla \overline{u_{\bar\chi,l}} - \bar \lambda \, \nu \, u_{\bar\chi,h} \, \overline{u_{\bar\chi,l}} \right) (\Psi \cdot n) d\sigma   \\
& \qquad\qquad\qquad\qquad\qquad\qquad-\int_{\Gamma_{t,\Phi_{\bar\chi}}} (\eps \nabla u_{\bar\chi,h} \cdot \nabla \overline{u_{\bar\chi,l}}) (\Psi \cdot n) d\sigma
\Bigg)_{h,l=1,\dots,m} .
\end{split}
\end{equation}
\end{itemize}
\end{cor}

The proof of Corollary  \ref{coroll:surfaceint} is an immediate application of Thorem \ref{theo:shapederivatives} and the following lemma.

\begin{lem}  \label{lem:helmsurfaceint}
Suppose that $v,w \in \L^{2}_{\nu}(\Omega_{\Phi_{\bar\chi}})$ are Helmholtz eigenfunctions associated with $0 \neq\bar \lambda$. Suppose that $v,w$ are $\H^2$-regular.
Then 
\begin{equation}
\begin{split}
&\scp{ \left(\p_{\Psi  }\eps+(\div\Psi)\eps-2\sym(J_{\Psi}\eps) \right) \na_{\ga_{\!t,\Phi_{\bar\chi}}} v}{\na_{\ga_{\!t,\Phi_{\bar\chi}}} w }_{\L^{2}(\Omega_{\Phi_{\bar \chi}})} \\
&\qquad \qquad  - \bar \lambda \scp{ \left( \p_{\Psi}\nu+(\div\Psi)\nu \right) v }{ w  }_{\L^{2}(\Omega_{\Phi_{\bar \chi}})} \\
& = \int_{\Gamma_{n,\Phi_{\bar\chi}}} \left( \eps \nabla v \cdot \nabla \overline{w} - \bar \lambda \,\nu \, v \, \overline{w} \right) (\Psi \cdot n) d\sigma -\int_{\Gamma_{t,\Phi_{\bar\chi}}} (\eps \nabla v \cdot \nabla \overline{w}) (\Psi \cdot n) d\sigma.
\end{split}
\end{equation}
\end{lem}

\begin{proof}
To ease the notation,  throughout this proof we write  $\Phi, \lambda$ instead of $\Phi_{\bar\chi}, \bar{\lambda}$, and we use the Einstein notation omitting the summation symbols.

First recall that $(\partial_\Psi \eps)_{kj} = \nabla \eps_{kj} \cdot \Psi = \Psi_i  \, \partial_i \eps_{kj}$ (cf. also formula \eqref{dzeta}) and that $-\div(\eps \nabla v)= \nu \, \lambda \, v$. Observe also that the quantities $\Psi, \nu, \eps$ have  real-valued entries, thus they coincide with their complex conjugates. The complex conjugate of $w$ is denoted by $\overline{w}$.

By integrating by parts we have
\begin{equation} \label{form:1}
\begin{split}
\int_{\omp} J_{\Psi}\eps \ \nabla v \cdot \nabla \overline{w} & = \int_{\omp}  (\partial_k \Psi_i) \eps_{kj} \, \partial_j v \, \partial_i \overline{w} = \int_{\Gamma_\Phi}  \Psi_i \, \eps_{kj} \, \partial_j v \, \partial_i \overline{w} \, n_k \, d\sigma -\int_{\omp} \Psi_i  \, (\partial_k \eps_{kj}) \, \partial_j v \, \partial_i \overline{w} \\
& \qquad - \int_{\omp} \Psi_i \, \eps_{kj} \, (\partial_k \partial_j v) \, \partial_i \overline{w}  - \int_{\omp} \Psi_i \, \eps_{kj} \, \partial_j v \, (\partial_k\partial_i \overline{w})  \\
&=\int_{\Gamma_\Phi}  (\Psi \cdot \nabla \overline{w})(\eps \, \nabla v \cdot n) d\sigma - \int_{\omp} (\Psi \cdot \nabla \overline{w}) \div(\eps \nabla v)  - \int_{\omp} \Psi_i \, \eps_{kj} \, \partial_j v \, (\partial_k\partial_i \overline{w})\\
&=\int_{\Gamma_\Phi}  (\Psi \cdot \nabla \overline{w})(\eps \, \nabla v \cdot n) d\sigma + \lambda \int_{\omp} (\Psi \cdot \nabla \overline{w}) \nu \, v  - \int_{\omp} \Psi_i \, \eps_{kj} \, \partial_j v \, (\partial_k\partial_i \overline{w}).
\end{split}
\end{equation}
Similarly, integrating by parts, exchanging the second-order derivatives, and then integrating by parts once again, we get
\begin{equation} \label{form:2}
\begin{split}
\int_{\omp} (J_{\Psi}\eps)^\top \ \nabla v \cdot \nabla \overline{w} & = \int_{\omp} J_{\Psi}\eps \ \nabla \overline{w} \cdot \nabla v \\
&=\int_{\Gamma_\Phi}  (\Psi \cdot \nabla v)(\eps \, \nabla \overline{w} \cdot n) d\sigma +  \lambda \int_{\omp} (\Psi \cdot \nabla v) \nu \, \overline{w}   - \int_{\omp} \Psi_i \, \eps_{kj} \, \partial_j \overline{w} \, (\partial_k\partial_i v) \\
&=\int_{\Gamma_\Phi}  (\Psi \cdot \nabla v)(\eps \, \nabla \overline{w} \cdot n) d\sigma +  \lambda \int_{\omp} (\Psi \cdot \nabla v) \nu \, \overline{w}   - \int_{\omp} \Psi_i \, \eps_{kj} \, \partial_j \overline{w} \, (\partial_i \partial_k v) \\
&=\int_{\Gamma_\Phi}  (\Psi \cdot \nabla v)(\eps \, \nabla \overline{w} \cdot n) d\sigma +  \lambda \int_{\omp} (\Psi \cdot \nabla v) \nu \, \overline{w}  + \int_{\omp} \div \Psi \, \eps \, \nabla \overline{w} \cdot \nabla v \\
& \qquad +\int_{\omp} \partial_\Psi \eps \, \nabla \overline{w} \cdot \nabla v  + \int_{\omp} \Psi_i \, \eps_{kj} \, (\partial_i \partial_j \overline{w}) \partial_k v -   \int_{\Gamma_\Phi} (\eps \, \nabla \overline{w} \cdot \nabla v) (\Psi \cdot n) d\sigma.
\end{split}
\end{equation}
Moreover, making  again use of the integration by parts we have that
\begin{equation}  \label{form:3}
\begin{split}
\int_{\omp} (\partial_\Psi \nu) \, v \, \overline{w} &= \int_{\omp}  \Psi_i \, \partial_i \nu \, v \, \overline{w} \\
&= -\int_{\omp} \div \Psi \, \nu \, v \, \overline{w} -\int_{\omp} (\Psi \cdot \nabla v) \, \nu \, \overline{w}  - \int_{\omp} (\Psi \cdot \nabla \overline{w}) \, \nu \, v + \int_{\Gamma_\Phi} \nu \, v \, \overline{w} \, (\Psi \cdot n).
\end{split}
\end{equation}
Hence, recalling that $\eps$ is a symmetric matrix and that the summation indices are mute, and using \eqref{form:1}, \eqref{form:2}, \eqref{form:3}, we obtain that
\begin{equation*}
\begin{split}
&\scp{ \left(\p_{\Psi  }\eps+(\div\Psi)\eps-2\sym(J_{\Psi}\eps) \right) \nabla v}{\nabla w }_{\L^{2}(\omp)}  -  \lambda \scp{ \left( \p_{\Psi}\nu+(\div\Psi)\nu \right) v }{ w}_{\L^{2}(\omp)} \\
&=\int_{\omp} \left(\partial_\Psi \eps + \div \Psi \, \eps - J_\Psi \eps - (J_\Psi \eps)^\top \right) \nabla v \cdot \nabla\overline{w} - \lambda \int_{\omp} \left(  \partial_\Psi \nu + \div \Psi \, \nu \right) v \, \overline{w}\\
&= \int_{\Gamma_\Phi} (\eps \, \nabla v \cdot \nabla \overline{w} -  \lambda \, \nu \, v \, \overline{w}) (\Psi \cdot n) d\sigma -\int_{\Gamma_\Phi} (\Psi \cdot \nabla \overline{w}) (\eps \, \nabla v \cdot n) d\sigma -\int_{\Gamma_\Phi} (\Psi \cdot \nabla v) (\eps \, \nabla \overline{w} \cdot n).
\end{split}
\end{equation*}
Recall now the boundary conditions, i.e. $v,w=0$ on $\Gamma_{t,\Phi}$ and $n \cdot \eps \, \nabla u = n \cdot \eps \, \nabla v=0$ on $\Gamma_{n,\Phi}$. Therefore on $\Gamma_{t,\Phi}$ we have that $\nabla v, \nabla w$ are parallel to the normal $n$, that is $\nabla v = (\nabla v \cdot n) n$ and $\nabla w = (\nabla w \cdot n)n$ on $\Gamma_{t,\Phi}$. Then it is easy to see that 
\begin{equation*}
\begin{split}
& \int_{\Gamma_\Phi} (\eps \, \nabla v \cdot \nabla \overline{w} - \lambda \, \nu \, v \, \overline{w}) (\Psi \cdot n) d\sigma -\int_{\Gamma_\Phi} (\Psi \cdot \nabla \overline{w}) (\eps \, \nabla v \cdot n) d\sigma -\int_{\Gamma_\Phi} (\Psi \cdot \nabla v) (\eps \, \nabla \overline{w} \cdot n)\\
&\quad =\int_{\Gamma_{n,\Phi}} \left( \eps \nabla v \cdot \nabla \overline{w} -  \lambda \,\nu \, v \, \overline{w} \right) (\Psi \cdot n) d\sigma -\int_{\Gamma_{t,\Phi}} (\eps \nabla v \cdot \nabla \overline{w}) (\Psi \cdot n) d\sigma.
\end{split}
\end{equation*}
\end{proof}

\begin{rem}   
If  $\Gamma_t$ and $\Gamma_n$ are separated (in other words $\Gamma=\Gamma_t\cup\Gamma_n$ 
and $\ol{\ga}_{t}\cap\ol{\ga}_{n}=\emptyset$),  $\Omega_\Phi$ is a bounded domain of  class $\C^{1,1}$ and the coefficients $\eps,\nu$ are of class $\C^1$  then  the eigenfunctions  of   the Helmholtz problem  on $\Omega_{\Phi}$   belong to $\H^2(\Omega_\Phi)$, see e.g., \cite{gilbarg, tro}.
\end{rem}

\section*{Acknowledgments}

The authors are thankful to Prof. Massimo Lanza de Cristoforis for useful discussions and for bringing to their attention  items \cite{LA1998},
\cite{LA2000} in the reference list. 
The first  named author acknowledges support by the project ``Perturbation problems and
asymptotics for elliptic differential equations: variational and potential theoretic methods'' funded by the MUR ``Progetti di Ricerca di Rilevante Interesse
Nazionale'' (PRIN) Bando 2022 grant no. 2022SENJZ3; he is also a member of the Gruppo Nazionale per l'Analisi  Matematica, la Probabilit\`{a} e le loro Applicazioni (GNAMPA) of the Istituto Nazionale di Alta Matematica (INdAM).
The third named author is very thankful for the kind 
invitation of the second author and the great hospitality of the 
Fakult\"at Mathematik, Institut f\"ur Analysis of the Technische Universit\"at Dresden
during his visiting period 
when some problems discussed in this paper where addressed.
He also acknowledges  financial support by ``Fondazione Ing. Aldo Gini'' for  his visit in Dresden.
His work was partially realized during his stay
at the Czech Technical University in Prague with the support of the EXPRO grant No.
20-17749X of the Czech Science Foundation. He has also received support from the French
government under the France 2030 investment plan, as part of the Initiative d’Excellence d’Aix-Marseille Université – A*MIDEX AMX-21-RID-012.


\bibliographystyle{plain} 


\bibliography{lpz}


\section{Appendix}

\subsection{More Details on the Proof of Theorem~\ref{mainpro}}

Let $F$ be the bilinear map from the product space $C(\Gamma , {\mathcal{L}}(V,V'_{\sharp})    )\times \C(\Gamma , {\mathcal{L}}(V'_{\sharp},V)   )$ to $\C(\Gamma , {\mathcal{L}}(V,V)    )$ defined by
\begin{equation}
\label{comp}
F(f,g)=g\circ f
\end{equation}
for all $(f,g)\in  \C(\Gamma , {\mathcal{L}}(V,V'_{\sharp})    )\times \C(\Gamma , {\mathcal{L}}(V'_{\sharp},V)   )$, where formula \eqref{comp} has to be understood as follows:
$$
(g\circ f)(\zeta ) = g(\zeta )\circ f(\zeta )
$$
 for all $\zeta \in \Gamma$.  Since the composition operator defines a bilinear continuous map, it follows that $F(f,g)$ is a well-defined element of 
 $\C(\Gamma , {\mathcal{L}}(V,V)        ) $. Moreover
 $$
 \|  (g\circ f)(\zeta )   \|\le \| g(\zeta )    \|  \| f(\zeta ) \|\le \|  g\| \|  f\| 
 $$
 for all $\zeta \in \Gamma$, hence
 $$
 \|  g \circ f   \| \le \|  g\| \|  f\| 
 $$
which implies that  $F$ is a bilinear continuous map, in particular it is analytic.  

Let  $f_0\in \C(\Gamma , {\mathcal{L}}(V,V'_{\sharp})    )$ be fixed and assume that it has the following property: for all $\zeta \in \Gamma$ the operator $f_0(\zeta )$ is invertible. Since the inversion operator which takes any invertible operator to its inverse is continuous (actually, it is analytic), it follows that the function  $f^{-1}_0$ from $\Gamma $ to $ {\mathcal{L}}(V'_{\sharp },V)   $ defined by 
$$
f^{-1}_0(\zeta )  =(f_0(\zeta ))^{-1}
$$ 
for all $\zeta \in \Gamma$, is an element of $\C(\Gamma , {\mathcal{L}}(V'_{\sharp },V)     )$ and satisfies the equation 
$$F(f_0, f^{-1}_0)=\mathcal I$$ where $\mathcal I$ is the constant function from $\Gamma $ to ${\mathcal{L}}(V,V)  $ defined by  $\mathcal I(\zeta )=I$ for all $\zeta \in \Gamma$, and $I$ is the identity operator. We now apply the Implicit Function Theorem to the function $F$ at the point $(f_0,f_0^{-1})$. Note that the differential of $F$ with respect to the variable $g$,  computed at $\psi$ (the linearised variable) is  given by the formula
$
d_gF_{|(f,g)=(f_0,f_0^{-1})} [\psi ]=\psi\cdot f_0 $ which shows that $d_gF_{|(f,g)=(f_0,f_0^{-1})}$ is bijective from $\C(\Gamma , {\mathcal{L}}(V'_{\sharp },V)     )  $ to 
$\C(\Gamma , {\mathcal{L}}(V,V)     )  $.
It follows that there exist a neighborhood $\mathcal U$ of $f_0$ in $\C(\Gamma , B(\H_0))$, a neighborhood $\mathcal V$ of $f_0^{-1}$ in $\C(\Gamma , B(\H_0))$
and an analytic map $\mathcal R$
 from $\mathcal U$ to $\mathcal V$ 
such that 
\begin{equation}\label{dini}
\{ (f,g)\in \mathcal U\times \mathcal V:\  F(f,g)=\mathcal I \}=\{ \mathcal (f,R (f)):\ f\in \mathcal U   \}\, .
\end{equation}
Consider now the curve $f_{\chi}(\zeta )=\TAc -\zeta \Jc$ defined for any  $\chi\in\mathcal{X}$ with $\| \chi-\bar \chi\| <\delta $. 
 For all $\zeta \in \Gamma$ the operator 
$\TAc-\zeta \Jc$ is invertible, hence $F(\TAc-\zeta \Jc, (\TAc-\zeta \Jc   )^{-1} )=\mathcal I$.  It follows that once we fix $f_0=f_{\chi}$ for some $\chi$ with $\|\chi -\bar \chi \|<\delta$, we can consider the corresponding map $\mathcal R$  and deduce from \eqref{dini} that $ (\TAc-\zeta \Jc   )^{-1} =\mathcal R ( \TAc-\zeta \Jc  )  $. Thus $\chi \mapsto (\TAc-\zeta \Jc   )^{-1}$ is of class $\C^r$ being the composition of the  analytic map $\mathcal R$  and the map   $\TAc-\zeta \Jc$ which is     of class $\C^r$.  Again, by our regularity assumption, we conclude that also  $ (\TAc-\zeta \Jc   )^{-1}\circ \Jc$ has the same regularity with respect to $\chi$, hence by formula \eqref{res1} the map 
$ \chi \mapsto  (\Tc-\zeta I  )^{-1}$ which takes values in  $\C (\Gamma,   {\mathcal{L}}(\H_0, \H_0)  )  $ is of class $\C^r$ or analytic. 

Finally, since the integration  on $\Gamma$  defines a linear and continuous map with respect to the integrating function, it follows that $\int_{\Gamma}  (\Tc-\zeta I   )^{-1}  d\zeta  $ is of class $\C^r$  with respect to $\chi$ for  $\| \chi-\bar \chi \|<\delta $.


\vspace*{5mm}
\hrule
\vspace*{3mm}


\end{document}